\documentclass[11pt,a4paper,reqno,pdflatex]{amsart}
\usepackage[margin=1in]{geometry} 
\usepackage{float}
\usepackage{amsfonts,amsmath,graphics,epsfig,latexsym,times}
\usepackage{ae,aecompl}
\newif\ifpix \pixtrue
\usepackage{setspace,fancyhdr,amssymb,amsthm,graphicx,ifthen,float,calrsfs} 
\usepackage{hyperref}
\usepackage{tikz,pgf}
\numberwithin{equation}{section} 
\usepackage[textsize=tiny]{todonotes} 
\newtheorem{thm}{Theorem}[section]  
\newtheorem{cor}[thm]{Corollary}    % Corollary environment
\newtheorem{lem}[thm]{Lemma}        % Lemma environment
\newtheorem{prop}[thm]{Proposition}  % Proposition environment
\theoremstyle{definition} \newtheorem{dfn}[thm]{Definition}

\newtheorem{rem}[thm]{Remark}  % Proposition environment
\newtheorem*{claim*}{Claim}

\allowdisplaybreaks[1]
\newcommand{\RR}{\mathbb{R}} 

\newcommand{\h}{\hbar}
\newcommand{\ord}{\operatorname{ord}}

\newcommand{\cR}{\mathcal{R}}
\newcommand{\cD}{\mathcal{D}}

\newcommand{\rd}{{\mathrm d}}

\newcommand{\p}{\partial}
\newcommand{\dbar}{\overline{\partial}}

\DeclareMathOperator{\supp}{supp}

\newcommand{\bC}{\mathbb{C}}

\newcommand{\cK}{\mathcal{K}}

\newcommand{\cO}{\mathcal{O}}

\newcommand{\ve}{\varepsilon}

\renewcommand{\geq}{\geqslant}
\renewcommand{\leq}{\leqslant}
\renewcommand{\ge}{\geqslant}
\renewcommand{\le}{\leqslant}
\renewcommand{\phi}{\varphi}
\renewcommand{\epsilon}{\varepsilon}
\author{Julius Ross}
\address{Department of Pure Mathematics and Mathematical Statistics,
  University of Cambridge}
\email{j.ross@dpmms.cam.ac.uk}
\author{Michael Singer}
\address{Department of Mathematics, University College London}
\email{michael.singer@ucl.ac.uk}
\date{\today}
\title{Asymptotics of Partial Density Functions for Divisors}
\begin{document}
\begin{abstract}
We study the asymptotic behaviour of the partial density function associated to sections of a positive hermitian line bundle that vanish to a particular order along a fixed divisor $Y$.  Assuming the data in question is invariant under an $S^1$-action (locally around $Y$) we prove that this density function has a distributional asymptotic expansion that is in fact smooth upon passing to a suitable real blow-up.  Moreover we recover the existence of the ``forbidden region'' $R$ on which the density function is exponentially small, and prove that it has an ``error-function'' behaviour across the boundary $\partial R$.  As an illustrative application, we use this to study a certain natural function that can be associated to a divisor in a K\"ahler manifold.
\end{abstract}
\maketitle
\section{Introduction} 
For motivation, consider as a toy example the space of polynomials on $\mathbb C^d$ with inner product
$$(p,q) := \int_{\mathbb C^d} p(z) \overline{q(z)} e^{-k|z|^2} d\lambda,$$
where $k\in \mathbb N$ and $d\lambda$ is $(2\pi)^{-d}$ times the
Lebesgue measure.    For given $\epsilon>0$, let $\{p_{\alpha}\}$ be
an orthonormal basis of polynomials that vanish to order at least
$\epsilon k$ along the hyperplane $\{z_1=0\}\subset \mathbb C^d$.
Then the \emph{partial density function} associated to these data is the smooth function
$$ \rho^\epsilon_{k} (z) := \sum_{\alpha} |p_{\alpha}(z)|^2 e^{-k|z|^2}.$$
This is independent of choice of orthonormal basis, and our interest
lies in its asymptotic behaviour as $k$ tends to infinity.  Using a
basis in which the basis $\{p_\alpha\}$ consists of monomials, after
straightforward calculation, 
\begin{equation}
  \label{eq:partialmodel}
 k^{-d}\rho^\epsilon_{k} = e^{-kx} \sum_{j\ge \epsilon k}
 \frac{(kx)^j}{j!},\;\; x = |z_1|^2.
\end{equation}
Regarding $\rho^\ve_k$ as a function of $x$, we find the asymptotic
behaviour 
\begin{equation}
   k^{-d}\rho^\epsilon_{k}(x) \sim \frac{1}{\sqrt{2\pi x}}\int_{-\infty}^{\sqrt{k}(x-\epsilon)} e^{-\frac{t^2}{2x}} dt\quad\text{ for }k\gg 0.\label{eq:partialgaussian}
 \end{equation}
This can be seen, for instance, through the Central Limit Theorem applied to  $k$ independent Poisson random variables with parameter $x$.  Thus $k^{-d}\rho^\epsilon_k$ is asymptotically a standard error-function centred at $x=\epsilon$;  in particular it tends to zero exponentially fast on the set $R=\{x<\epsilon\}$, and tends to 1 on $\{x>\epsilon\}$ as $k$ tends to infinity.\medskip

In this paper we study the analogous partial density function associated to sections of high powers of a positive hermitian line bundle that vanish to a particular order along a fixed divisor.   Similar density functions have found a wide range of uses, including the study of random matrices (e.g.\ Shiffman--Zelditch \cite{Shiffman}, Berman \cite{BermanCn}), in K\"ahler geometry (e.g.\ Pokorny--Singer \cite{Pokorny}, Ross--Witt-Nystr\"om \cite{Ross2}) and in dimension $d=1$ is closely related to the Laplacian growth (e.g.\ Hedenmalm-Makarov \cite{Hedenmalm2,Hedenmalm}).  \medskip

In the work of Shiffman--Zelditch \cite{Shiffman} it is shown,
essentially in the toric case, that there is a subset in which the
partial density function is exponentially small (and it is here that this is given the name
``forbidden region'').   Through work of Berman \cite{Berman} it is
known (at least when the base is compact) that there is again an open
forbidden region $R$ containing the divisor such that asymptotically
the partial density function is exponentially small on compact subsets
of $R$ and is equal to the usual density function on compact subsets
of the complement of $\overline{R}$.   This has been studied again in
detail in the toric case \cite{Pokorny} but other than this rather
little is known about the behaviour of the partial density function
near the boundary of $R$.\medskip 

Our results below give an essentially complete description of the
partial density function when all the data in question are invariant
under a local holomorphic $S^1$-action and $\epsilon$ is sufficiently
small.  Roughly speaking it states that there is a natural way in
which the partial density function has a globally defined asymptotic
expansion in powers of $k^{1/2}$ whose terms depend on the curvature
of the hermitian metric, all of which are (in principle) computable.
Moreover,  by working out the leading term we recover the existence of
this forbidden region and show that the partial density function has
the same error-function behaviour across its boundary as it does in
the model case \eqref{eq:partialgaussian}.\medskip 

Before stating precise theorems we return once more to the toy example
above.  A convenient way to express the existence of an asymptotic
expansion is through the semi-classical variable $\h = k^{-1/2}$. Then
to say that a function of the form $k^{-d}\rho^\epsilon_k(x)$ for
$x\in \mathbb R_{>0}$  admits a smooth asymptotic expansion in powers
of $k^{1/2}$ is to say that the function 
$$ \hat{\rho}(\h,x) := k^{-d}\rho^\epsilon_k(x) \text{ for } \h = \frac{1}{\sqrt{k}} \text{ and } k\in \mathbb N$$
extends to a smooth function on $\mathbb R_{\ge 0}\times \mathbb R_{>0}$, i.e.\ it extends to a smooth function right up to the boundary on which $\h=0$.  Now if $\rho^\epsilon_k$ is the partial density function in \eqref{eq:partialmodel} then $\hat{\rho}$ cannot extend smoothly to the entire boundary $\{\h=0\}$. For as we have seen, its leading order term in $\hbar$ is
$$\frac{1}{\sqrt{2\pi x}}\int_{-\infty}^{\frac{x-\epsilon}{\h}} e^{-\frac{t^2}{2x}} dt$$
which does not extend smoothly over the point $(x,\h)=(\epsilon,0)$ due to the presence of the term $\xi: = \frac{x-\epsilon}{\h}$. However, we can formally circumvent this by considering $\hat{\rho}$ instead as a function of $\xi$ and $\h$, at which point its smoothness is immediate.  This is made precise by an approach advocated by Melrose:   we consider instead the lift of $\hat{\rho}$ to the real blow-up of $\mathbb R_{\ge 0}\times \mathbb R_{>0}$ at the point $(0,\epsilon)$ which does extend to a smooth function across the boundary.\medskip

With this in mind we state our main results.  Let $X$ be a compact complex manifold and $L$ be a holomorphic line bundle with a positive smooth hermitian metric $h$.      These induce an  $L^2$-inner product on the space of sections $H^0(L^k)$ and if $Y$ is a smooth divisor in $X$ then for $\epsilon\ge 0$ the partial density function is defined to be
$$\rho_{k}^{\epsilon} = \sum_{\alpha} |s_{\alpha,k}|_{h^k}^2$$
where $\{s_{\alpha,k}\}$ is an $L^2$-orthonormal basis for the space of holomorphic sections of $L^k$ that vanish to order at least $\epsilon k$ along $Y$.   When $\epsilon=0$ we denote this simply by $\rho_k$ which is the usual density function (often called the Bergman function).    We suppose that there is a neighbourhood $U$ of $Y$ that admits a holomorphic $S^1$-action on $X$ which is standard in local coordinates around $Y$.  That is,
\begin{equation}\label{e6.9.12.12}
e^{i\theta}\cdot(z,w) = (e^{i\theta}z,w)\text{ for }e^{i\theta}\in S^{1}
\end{equation}
where $Y$ is locally defined by $z=0$.  We also assume that the
restrictions of all the initial data to $U$ (i.e.\ the line bundle and
metric) are invariant under this action. Denote by $\mu\colon X\to
\mathbb R$ the Hamiltonian of the action, normalized so that
$\mu^{-1}(0) = Y$, and let $v$ be the vector field generating the
$S^1$-action. 
\begin{thm} \label{thm:main1} Given the above data, for sufficiently small $\epsilon$,
  we have
\begin{equation}\label{e7.9.12.12}
\rho_{k}^{\epsilon} \sim \left\{\begin{array}{ll} O(k^{-\infty})&
    \mbox{ on }\mu^{-1}[0,\epsilon), \\
\rho_{k} + O(k^{-\infty})&\mbox{ on }
X\setminus\mu^{-1}[0,\epsilon].\end{array}\right.
\end{equation}
\end{thm}
By this statement we mean that equality holds on any given compact subset of $\mu^{-1}[0,\epsilon)$ (respectively $X\setminus\mu^{-1}[0,\epsilon]$).  Thus $\mu^{-1}[0,\epsilon)$ is precisely the forbidden region described above.  

\begin{thm}\label{thm:main2}
Let $\rho^\epsilon_{k}$ be the partial density function and  set
$$\hat{\rho}(\hbar,x)  = k^{-d}\rho^\epsilon_k \text{ for } \hbar = k^{-1/2}.$$    Then  $\hat{\rho}$ has a distributional asymptotic
expansion on $X$.  In fact this distribution is the push-forward
by the blow-down map $\beta$ of a smooth function on the real blow-up of $X\times [0,\infty)$ along $\mu^{-1}(\epsilon) \times \{0\}$.
%\begin{equation}\label{e8.9.12.12}
%\beta\colon [ X \times [0,\infty) ; \mu^{-1}(\epsilon)\times \{0\}]\to X\times [0,\infty]
%\end{equation}
Its leading order term is given by
\begin{equation}\label{e1.13.2.15}
\hat{\rho}(z,\hbar) =
\frac{1}{\sqrt{2\pi|v(z)|^2}}\int_{-\infty}^{\frac{\mu(z)-\epsilon}{\hbar}}
e^{-\frac{t^2}{2|v(z)|^2}} dt + O(\hbar)
\end{equation}
for $(z,\hbar)$ such that 
$$
\frac{\mu(z)-\ve}{\hbar}
$$
is bounded.  
\end{thm}

We refer the reader Appendix \ref{appendix:blow-up} for a summary of this real-blowup that is denoted $[ X \times [0,\infty) ; \mu^{-1}(\epsilon)\times \{0\}]$.   Roughly speaking it is obtained by replacing the submanifold $\mu^{-1}(\epsilon)$ with a ``half-cylinder'' as in the following picture, and allows the use of well-defined ``polar coordinates'' centered at points in $\mu^{-1}(\epsilon)$.  

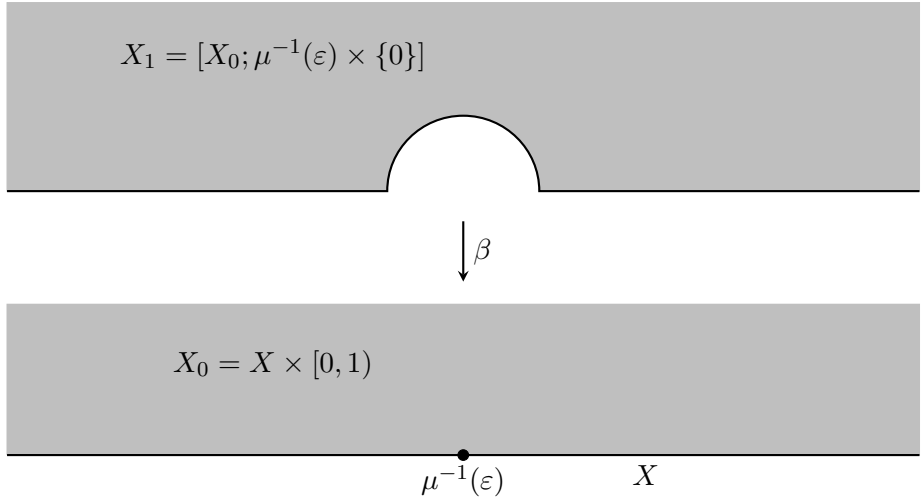
\begin{center}
\begin{tikzpicture}[>=stealth,domain=-4:4,smooth,scale=1]%,rotate=90, xscale = -1]
%
%  Shading in the blow-up
%
\filldraw[color = lightgray] 
(1,0) arc (0:180:1) -- (-6,0) --
(-6,2.5) -- (6,2.5) --  (6,0mm) -- (1,0);
\draw (-2.5,1.8) node  {$X_1 = [X_0;
  \mu^{-1}(\epsilon)\times \{0\}]$};
%\draw[very thin,color=gray] (-0.1,-1.1) grid (3.9,3.9);
%\draw[thick] (-4,0) -- (-1.,0) node[right] {$x$}; 
%\draw[thick]  (1,0) -- (4,0) node[above] {somthing};
%\draw[color=red]	plot (\x,{1.2/(1+\x*\x/2.88)});
\draw[thick, color=black] (6,0) --(1,0) arc (0:180:1) -- (-6,0);
%\draw[thick,color=magenta] (0,1) -- (0,2.5); 
%\draw[thick, color = black, domain = -0.9999:0.9999,smooth] plot(\x, {(1
%  - \x*\x)^{0.5}})  ;
\filldraw[color = lightgray]  [yshift = -3.5cm]
(-6,0) --
(-6,2) -- (6,2) --  (6,0mm) -- (-6,0);
%\draw [yshift=-3.5cm,thin,red] (-4,0.2) -- (4,0.2);
\draw [yshift=-3.5cm](-2.5,1.2) node  {$X_0 = X \times [0,1)$};
\draw[thick, color = black] [yshift=-3.5cm] (-6,0) -- (6,0);
\draw (2.4,0) [yshift = -3.5cm]node[below]{$X$};
\filldraw (0,0) [color=black,yshift = -3.5cm] circle (2pt) node[below]{$\mu^{-1}(\epsilon)$};
%\draw[thick,color=magenta]  [yshift = -3.5cm] (0,0) -- (0,2);
%\draw[color=magenta] (0,-2.3) node[right] {$\{0\}\times I$};
%
%  $I$ interval
%
% \draw[thick,color = black] [xshift = 5.1cm] (0,0) -- (0,2.5);
% \draw[color = black] [xshift = 5.1cm] (0, -0.05) node[right]{$\epsilon = 0$};
% \filldraw[color = black] [xshift = 5.1cm] (0,0) circle (2pt);
% \filldraw[color = red] [xshift = 5.1cm] (0,0.2) circle (2pt);
% \draw[color = red] [xshift = 5.1cm] (0,0.35) node[right]{$\epsilon =
%   0.2$};
% \filldraw[color = black] [xshift = 5.1cm,yshift = -3.5cm] (0,0) circle (2pt);
% \filldraw[color = red] [xshift = 5.1cm,yshift=-3.5cm] (0,0.2) circle (2pt);
% \draw [xshift = 5.1cm,yshift=-3.5cm] (0,0) node[right]{$\epsilon =
%   0$};
% \draw[color = black] [xshift = 5.1cm] (0,2.45) node[right]{$I$};
% \draw[color = black] [xshift = 5.1cm,yshift = -3.5cm] (0,2.45)
% node[right]{$I$};
% \draw[color = black,thick] [xshift = 5.1cm,yshift = -3.5cm] (0,0) --
% (0,2);
%\draw[color = red,yshift=-3.5cm, xshift = 5.1cm] (0,0.35) node[right]{$\epsilon =
%  0.2$};
% %
%  Mappings
%
% Blow-down
\draw[->] [thick] (0,-.4) -- (0,-1.2);
\draw (0, -0.8) node[right] {$\beta$} ;
%
% phi
%
% \draw[->] [thick] (4.2,1.25) -- (5.0, 1.25);
% \draw (4.6, 1.25) node[above] {$\pi_1$} ;
% \begin{scope} [yshift = -3.5cm]
% \draw[->] [thick] (4.2,1.25) -- (5.0, 1.25);
% \draw (4.6, 1.25) node[above] {$\pi_0$} ;
% \end{scope}
%
%  Labelling faces of blow-up
%
%\draw (2.5,0) node[below] {$X_0 = [M;p]$};
%\draw (-2.1,-0.5) node {\framebox{$E$}};
%\draw[->,thin] (-.74,-0.18) .. controls (-0.1,0.5) and (0.15,0.1)
%.. (0.55,0.7);
%
%  Two points on the `exceptional divisor'
%
%\filldraw[color=blue] (30:1) circle (2pt);
%\filldraw[color=green] (70:1) circle (2pt);
%
% Labelling the corners
%
%\draw (0:1.1)  node [below]{$y=+\infty$} ;
%\filldraw (180:1.1) node [below]{$y=-\infty$};
\end{tikzpicture}
\end{center}

So using again the substitution $\hbar =1/\sqrt{k}$ we can interpret this as the statement that $\rho^{\epsilon}_{1/\hbar^2}$ is a smooth function on the half-space $X\times [0,\infty)$ (or really the restriction of such a smooth function for values of $\hbar$ with $\hbar^{-2}$ a positive integer).  In fact, there is almost certainly an interpretation of density functions for all real powers of $L$ using Melrose's ideas \cite{Melrose2}, which would give a natural interpretation of the (partial) density function for all values of the semi-classical parameter $\hbar$.\medskip

The idea of the proof is as follows.  Consider first the case that $X$
is one dimensional.  Then looking locally we are essentially interested in
the case $X$ is the unit disc with coordinate $z$ and $Y$ is the
origin.  Since all the data are $S^1$-invariant, the set of monomials
$z^{\epsilon k}, z^{\epsilon k +1},\cdots$ give an orthogonal set of
local sections that vanish to the right order along $Y$.   Thus the
partial density function can be calculated by summing the pointwise norm
of these functions, once they have been normalised to have unit $L^2$-length.
Using the moment-variable (given by a Legendre transform of the
potential defining the metric) this can be done with a combination of
standard techniques, namely Laplace's method to calculate the
integrals defining the length of these functions,  and then a combination
of Laplace's method and the Euler-Maclaurin formula to expand the resulting sum in powers of $k$. This gives a local formula for a quantity that ought to be the partial density function, and one argues that this is in fact
the case up to insignificant terms.   

Our real interest is the case of higher dimension.  Here we argue
similarly, replacing  the powers of $z$ with powers of a choice of defining section for $Y$. We then use a parameterized version of the Legendre transform to
deal with the directions normal to $Y$.  In fact, we interpret this Legendre transform as giving a family of hermitian metrics on certain line bundles on $Y$ which, along with a family of volume forms that we construct, gives a family of density functions on $Y$ (that play the role of the normalising of the sections $z^{j}$ to have unit length).  Together these can be summed to give a quantity that ought to be the partial density function, and again one proves this is the case up to insignificant terms.  Then, using essentially the same techniques as in the one-dimensional case, we prove that it has the desired asymptotic expansion in powers of $k$.  We refer the reader to Section \ref{sec:cylinder} for a more
detailed summary.\\

\noindent{\bf Acknowledgements:}   During this work JR was supported
by an EPSRC Career Acceleration Fellowship (EP/J002062/1) and MS  by
the Leverhulme Trust.  The authors wish to thank Bo Berndtsson, David
Witt Nystr\"om, H\r{a}kan Hedenmalm and Steve Zelditch for helpful
conversations. \\

\noindent{\bf Added in Proof:} The reader interested in this topic should be aware of recent work of Zelditch-Zhou \cite{ZZ} who have since proved similar results on the interface behaviour of the partial density function.

\section{Partial Bergman kernels}

The density function of a hermitian line bundle is the restriction to
the diagonal of the Bergman kernel, which is the reproducing kernel
for the $L^2$-projection to the space of holomorphic sections.  Here
we shall discuss the simple extension of this concept in which we
impose a certain vanishing of the sections along a fixed submanifold.   

\subsection{Definition of Partial Bergman Kernels}

We start by recalling some standard notation and terminology.  Let $L$
be a holomorphic line bundle on a complex manifold $X$ of dimension $d$, and $h$ be a
hermitian metric on $L$.      Given a local holomorphic trivialization
$\zeta$ of $L$ we can write $|\zeta(z)|_h^2 = e^{-2\phi_\zeta(z)}$ for
some smooth potential function $\phi_\zeta$.  By standard abuse of
notation we let $\phi$ denote this potential (even though it is not
globally defined) and we confuse $\phi$ with the metric $h=e^{-\phi}$.
Thus $h^k = e^{-k\phi}$ is the induced hermitian metric on
$L^k:=L^{\otimes k}$ for $k\in \mathbb N$.    In terms of transition functions, if $\zeta_\alpha = \lambda_{\alpha\beta} \zeta_\beta$ then $\phi_{\zeta_\beta} = \phi_{\zeta_\alpha} + \log |\lambda_{\alpha\beta}|$.  From this one can extend the notion of a potential to include the case of $\mathbb Q$-line bundles.  The abuse of
notation $h = e^{-\phi}$ also leads us to use the terminology
`$\phi$ is a (plurisubharmonic potential) for $L$' instead of $h$ is a
metric on $L$ (with positive curvature). 

If $s$
is a section of $L$ then by abuse of notation we let $s$ also denote its local
representative in a given trivialization: thus $|s(z)|_{\phi} =
|s(z)|e^{-\phi(z)}$.   The associated curvature form\footnote{Note the
  factor of $2\pi$}
\begin{equation}\label{e1.29.9.15}
\omega_{\phi}=dd^c\phi = \frac{1}{2\pi}d J d \phi =
\frac{i}{\pi}\partial\dbar \phi
\end{equation}
is a well defined $(1,1)$-form and the
hermitian metric $e^{-\phi}$ is said to be (strictly) 
positive if $\omega_{\phi}$ is a (strictly) positive form, which
occurs if and only if $\phi_\zeta$ is a (strictly) plurisubharmonic
function for all local trivializations $\zeta$.   When $\phi$ is
strictly plurisubharmonic we let $\omega_{\phi}^{[d]} = \frac{1}{d!}
\omega_{\phi}^d$ be the associated volume form.

The hermitian metric $e^{-\phi}$ gives a pairing $(\cdot,\cdot)_{\phi}
\colon L_z \otimes L_z \to \mathbb C$ that is linear in the first
variable and conjugate linear in the second.       The induced
$L^2$-inner product on two smooth sections $s,t$ of $L$ is given by  
$$ \langle s,t\rangle_{\phi} := \int_X (s,t)_{\phi} \omega_\phi^{[d]} = \int_X s(z)
\overline{t}(z) e^{-2\phi(z)} \omega_\phi^{[d]}$$ 
and we shall write $\|s\|_{\phi}$ for the corresponding $L^2$-norm.
We denote by $L^2_{k\phi}:=L^2_{k\phi}(X)$ the space of sections of
$L^k$ with finite $L^2$-norm, and let $H_{k\phi}:=H_{k\phi}(X)\subset
L^2_{k\phi}$ denote the subspace of holomorphic sections. 

We write $\overline{L}$ for the bundle $L$ with the conjugate complex structure.  If $V,W$ are auxiliary holomorphic vector bundles on $X$ then we let $(\cdot,\cdot)_{\phi}$ also denote the naturally induced pairing  $(L \otimes V)\otimes (L\otimes W)\to V\otimes \overline{W}$ and extend our other notations accordingly. \medskip

Assume now that $\phi$ is a smooth strictly plurisubharmonic potential on $L$. Since we will be interested in sections of $L^k$ asymptotically as $k$ tends to infinity we make this part of our definition of the partial Bergman kernel as follows.  Let $Y\subset X$ be a compact complex submanifold and  $\epsilon\in \mathbb Q_{\ge 0}$. 

\begin{dfn}
 Denote by $H^{\ve k}_{k\phi}$ the subspace of holomorphic sections of $L^k$ vanishing to order at least $\ve k$ along $Y$. 
\end{dfn}

\begin{dfn}\label{def:partialbergman}
The \emph{partial Bergman kernel} (PBK) for $(\epsilon,Y)$ is a
sequence of sections $K^{\epsilon}_k$ of $\overline{L}^k\boxtimes L^k$
on $X\times X$ for $k,\epsilon k \in \mathbb N$ such that 
  \begin{enumerate}
  \item For each fixed $z$, the section $z'\mapsto K^{\epsilon}_{k,z}(z'):= K^{\epsilon}_{k}(z,z')$ is in $\overline{L}^k_z\boxtimes H_{k\phi}^{\epsilon k}$ and
  \item We have
  \end{enumerate}
  \begin{align}\label{eq:reproducing}
 s(z) &= \langle s,K^{\epsilon}_{k,z}\rangle_{k\phi}=\int_X s(z') \overline{K}^{\epsilon}_k(z,z')e^{-2k\phi(z')} \, \omega^{[d]}_{\phi,z'}\text{ for all } z\in X \text{ and } s\in H_{k\phi}^{\epsilon k }
  \end{align}
where the notation $\omega^{[d]}_{\phi,z'}$ indicates the integral is being taken with respect to the variable $z'$.
\end{dfn}

\begin{dfn}
The \emph{partial density function} (PDF) $\rho^{\epsilon}_k$ is the norm of the restriction of the PBK to the diagonal, i.e.\
\begin{equation}
  \rho^{\epsilon}_{k}(z) := |K^{\epsilon}_k(z,z)|_{k\phi} = K^{\epsilon}_k(z,z)e^{-2k\phi(z)}
\end{equation}
which is a smooth real valued function on $X$.    
\end{dfn}

Of course if $\epsilon=0$ then this definition recovers what is commonly referred to as the \emph{Bergman kernel} and \emph{density function} for $L^k$ and we write these simply as $K_k$ and $\rho_k$.  We remark that all of these definitions can equally be made with $\omega_\phi^{[d]}$ replaced by a given smooth volume form $dV$ on $X$.   \medskip

\subsection{Existence and Basic Properties}
\label{sec:existence}
The existence of the PBK follows from standard
functional analysis considerations.    Suppose that $X$ is either a
compact K\"ahler manifold, or else a bounded domain, with smooth
boundary, in a K\"ahler manifold. 

As is well known, $H_{k\phi}^{\epsilon k }$ is a closed subspace of $L^2_{k\phi}$, so in particular it is a Hilbert space.  Moreover, if
$z\in X$, the evaluation map $s\mapsto s(z)$ is a bounded linear functional
$H^{\epsilon k}_{k\phi} \longrightarrow L^k_z$.  So, by the Riesz representation
theorem, there is an element
\begin{equation}\label{e1.5.6.14}
K^{\epsilon}_{k,z} \in \overline{L}^k_z\otimes H_{k\phi}^{\epsilon k } \text{ with } f(z) = \langle f,K^{\epsilon}_{k,z}\rangle_{k\phi}\text{ for all
}f\in H_{k\phi}^{\epsilon k}.
\end{equation}
Clearly $\langle u,K^{\epsilon}_{k,z}\rangle_{k\phi}=0$ if $u\in L^2_{k\phi}$ is orthogonal to $H^{\epsilon k }_{k\phi}$, so the kernel
\begin{equation}\label{e2.5.6.14}
K_k^{\epsilon }(z,z') := K^{\epsilon}_{k,z}(z')
\end{equation}
is the Schwartz kernel of the orthogonal projection
$\cK_k^{\epsilon}:L^2_{k\phi} \to H^{\epsilon k }_{k\phi}$ in the sense that
\begin{equation}\label{e1.8.6.14}
(\cK_{k}^{\epsilon}u) (z) = \langle u,K_{k,z}^{\epsilon}\rangle_{k\phi} = \int_{X} u(z')\overline{K_k^{\epsilon}(z,z')}e^{-2k\phi(z')}\omega^{[d]}_{z'} \text{ for } u\in L_{k\phi}^2.
\end{equation}
In particular this gives uniqueness of $K_k^{\epsilon}$.   Because $\cK^{\epsilon}_k$ is orthogonal, it is self-adjoint, and so \eqref{e2.5.6.14} is hermitian, i.e.\
\begin{equation}\label{e3.5.6.14}
\overline{K^{\epsilon}_{k,z'}(z)} = K^{\epsilon}_{k,z}(z').
\end{equation}
Thus we can also write \eqref{e1.8.6.14} as
\begin{equation}
\cK_{k}^{\epsilon}u (z) = \int_{X} u(z')K^{\epsilon}_k(z',z)e^{-2k\phi(z')}\omega^{[d]}_{z'}.
\end{equation}
In particular
\begin{equation}\label{e4.5.6.14}
\dbar_{z'} K^{\epsilon}_{k,z}(z') =0\mbox{ and }
\dbar_z \overline{K^{\epsilon}_{k,z}(z')} =0.
\end{equation}
%so $K_{k}^{\epsilon}$ is holomorphic as a section of $\overline{L}^k\boxtimes L^k$.  

\begin{rem}  If $X$ is a compact K\"ahler manifold the spaces $H_{k\phi}^{\epsilon k}$ are
  finite-dimensional and the definition above agrees with that in the introduction. For if $s_{\alpha,k}$ is an $L^2$-orthonormal basis for $H_{k\phi}^{\epsilon k} = H^0(L^k\otimes \mathcal I_Y^{\epsilon k})$ then one easily checks that
\begin{equation}\label{eq:Bergmansections}
K_k^{\epsilon}(z,z') = \sum_{\alpha} \overline{s_{\alpha,k}(z)} \boxtimes s_{\alpha,k}(z')
\end{equation}
has the characteristic properties of the PBK, so
\begin{equation}
  \rho^{\epsilon}_k = \sum_{\alpha} |s_{\alpha,k}|^2_{k\phi}.
\end{equation}
\end{rem}

\begin{rem}\label{rem:explicit} When $\epsilon=0$ the difference $v:=u - \cK_k[u]$ is the
$L^2_{k\phi}$-minimal solution of the equation $\dbar v = \dbar u$.  This follows
directly from the orthogonality of the projection.
\end{rem}

\subsection{Bounds for Bergman Kernels}

For convenience we recall here some basic estimates related to Bergman
kernels.   The fundamental `$L^2$ implies $L^{\infty}$ bound' for holomorphic
 functions (Proposition~\ref{prop:Bochner} below) allows us to
 estimate the norm of the evaluation map and hence by the discussion
 in \S\ref{sec:existence}, the norms of the PBK and PDF themselves.  These estimates will be refined using {\em extremal
   envelopes} in the next section.

 On a K\"ahler manifold $X$ we denote by $B_{z}(\delta)$ the geodesic
 ball of radius $\delta$ centred at $z$ (taken with respect to the
 given K\"ahler metric).   As usual $L$ will be a holomorphic line
 bundle on $X$ with smooth strictly plurisubharmonic potential $\phi$. 

\begin{prop}\label{prop:Bochner}
Let $X$ be a complex manifold and $W\subset X$ be relatively compact.  Then there is a constant $C_{W}$ such that for all $k$ sufficiently large
\begin{equation}\label{e3.18.6.14}
|f(z)|_{k\phi} \leq C_{W} k^{d/2} \left(\int_{B_z(k^{-1/2})}
  |f|_{k\phi}^2 \omega^{[d]}\right)^{1/2}
\le C_{W}k^{d/2}\|f\|_{k\phi}
\end{equation}
for all $f\in H_{k\phi}$  and all $z\in W$.
\end{prop}
\begin{proof}
For $z\in W$ we may choose local coordinates on a chart so that $z=0$, and by a change of gauge that $\phi(z) = O(|z|^2)$.  In this chart the geodesic metric $\rho(\cdot,\cdot)$ is equivalent to the Euclidean metric,  so there is a constant $c>0$ such that $c|z-z'|\le \rho(z,z')\le c^{-1}|z-z'|$ for points $z,z'$ in this chart.
 
Let $\chi(t)\colon \mathbb R\to [0,1]$ be a smooth non-negative cut-off function equal to $1$ for $t\leq 1/2$ and equal to $0$ for $t\geq
1$ and set 
$$\chi_k(t) = \chi( c^{-1} k^{1/2} t)$$   
We recall the {\em Bochner--Martinelli--Koppelman
  formula} \cite[Ch.\ IV]{Range}
\begin{equation}\label{e1.29.6.14}
u(z) = - \int_{\bC^d} \dbar u(w) \wedge B(w,z),\;\; u \in  C^\infty_0(\bC^d)
\end{equation}
where the Bochner--Martinelli kernel is given by the formula
\begin{equation}\label{e2.29.6.14}
B(w,z):=c_d\frac{1}{|w-z|^{2d}}\eta(w,z)
\end{equation}
where $c_d$ is a universal constant and
\begin{equation}
\eta(w,z) = \iota_{(\bar{w}-\bar{z})\dbar_{w}}\rd \bar{w}_1\wedge
\cdots \rd \bar{w}_d \wedge 
\rd w_1\wedge \cdots \rd w_d.
\end{equation}

We apply this with $u = \chi_k(|z|)f(z)$.   Then $\dbar
u = f \dbar \chi_k$ is supported in a spherical shell of radius $O(k^{-1/2})$
and $\dbar \chi_{k}=O(k^{1/2})$.   So
\begin{equation}\label{e3.29.6.14}
\left|\dbar (\chi_{k})B(w,0)\right| = O(k^d).
\end{equation}
Hence \eqref{e1.29.6.14} gives
\begin{equation}\label{e2.7.6.14}
f(0) = - \int_{|w|\le ck^{-1/2}} f(w) \overline{\partial} \chi_k B(w,0)
\end{equation}
where $k$ is to be taken large enough so the support of $\chi_k$ lies in this chart.   We are assuming $\phi(0)=0$, so $|f(0)|^2 = |f(0)|^2_{k\phi}$.   Thus using the Cauchy-Schwarz inequality,
\begin{equation}
|f(0)|_{k\phi}^2 \le C\int_{|w|\le c k^{-1/2}} |f(w)|_{k\phi}^2 \omega^{[d]} 
\int e^{k\phi}\left| \dbar \chi_{k} B(w,0) \right|^2,
\end{equation}
as $\omega^{[d]}$ is smooth so is equivalent to the euclidean volume form.  Now in the first integral we may replace the Euclidean ball $\{ |w|<c k^{-1/2}\}$ by the geodesic ball $B_0(k^{-1/2})$ and only improve the inequality.    On the other hand,  $k\phi = O(k|z|^2)$ on this chart so $e^{k\phi}$ is uniformly bounded  on the support of $\chi_k$ and the volume of this support is $O(k^{-d})$.    Thus \eqref{e3.29.6.14} implies the second integral is of order $O(k^d)$.  This proves the result for a single point, and the uniform estimate follows an obvious covering argument.
\end{proof}

Recall that $K^{\epsilon}_{k,z}\in \overline{L}^k_z\otimes H^{\epsilon k}_{k\phi}$ so with our convention of using the same notation for a section of a line bundle and its representation in a local frame we have
$$ \| K^{\epsilon}_{k,z}\|^2_{k\phi} = \int_X |K^{\epsilon}_{k,z}(w)|^2 e^{-2k(\phi(w) + \phi(z))} \omega_w^{[d]}.$$

\begin{cor}\label{cor:l2boundbergman}
Suppose $X$ is compact.  Then there is a constant $C$ such that 
\begin{equation}\label{e2.14.2.15}\| K^{\epsilon}_{k,z}\|_{k\phi}\le
  C k^{d/2}
\end{equation}
for all $z\in X$ and all  $\epsilon \ge 0$.
\end{cor}
\begin{proof}
Essentially by definition,  $\|K^{\epsilon}_{k,z}\|_{k\phi}$ is equal to the operator norm of the evaluation map $H_{k\phi}^{\epsilon k}\to L^k_z$ given by $s\mapsto s(z)$.  But the previous theorem says precisely that this operator norm is bounded by $C k^{d/2}$.
\end{proof}

\begin{cor}\label{cor:pointwiseboundbergman}
Suppose $X$ is compact.  Then there is a constant $C$ such that
\begin{equation}\label{e3.14.2.15}
|K^{\epsilon}_{k}(z,z')|_{k\phi} \leq C k^d \text{ for all } z',z\in X.
\end{equation}
\end{cor}
\begin{proof}
Apply the above to the function $y\mapsto K_{k,z}(y)$ itself to deduce that $|K_{k,z}(y)|_{k\phi}$ is bounded by $O(k^{d/2})\|K^{\epsilon}_{k,z}\|_{k\phi} = O(k^d)$.
\end{proof}

\subsection{Decay away from the Diagonal}

We wish to discuss the well known fact that the Bergman
kernel decays exponentially fast away from the diagonal  (see, for
instance,  \cite[Proposition 9]{Lindholm}, \cite[Prop 4.1]{Ma},
\cite{Delin}, \cite[Thm 0.1]{MM2}).  We let $\rho(z,z')$ denotes the geodesic distance
between points $z,z'\in X$ taken with respect to a given K\"ahler
metric, and in this section we assume $X$ is compact.

\begin{thm}[Decay away from the diagonal]\label{thm:decaynonpartial}
The Bergman kernel $K_k$ decays exponentially fast away from the diagonal, in the following sense:  there are constants $C,c>0$ such that
$$|K_k(x,y)|_{k\phi} \le Ck^{d} e^{-c \sqrt{k}\rho(x,y)} \text{ for all } x,y\in X.$$
\end{thm}

For convenience of the reader we give a proof of this fact, based on a
variant of the ``Donnelly-Fefferman trick'' essentially due to Bo
Berndtsson \cite{Bo} (and whom we thank for pointing us in this
direction). 

\begin{thm}\label{thm:Bo}
Let $f$ be a $\overline{\partial}$-closed $(0,1)$ form with values in $L$ and $v$
be the $L^2$-minimal solution to the equation $\overline{\partial} v =
f$ and let $x\in X$.    Then there are constants  $c,C$ (independent of $f,k,x$) such that
$$ \int_X |v(z)|^2_{k\phi}  e^{-2c\sqrt{k}\rho(z,x)} \omega_z^{[d]} \le \frac{C}{k} \int_X |f(z)|^2_{k\phi} e^{-\frac{c}{2}\sqrt{k}\rho(z,x)} \omega_{z}^{[d]}.$$
In particular if $f=\overline{\partial}g$ for some $g$ then this holds for $v=g-\cK_k(g)$.
\end{thm}

We remark that the above statement may well be suboptimal but is sufficient for our purpose (the statement in \cite[Theorem 2.1]{Bo} replaces the terms $2c$ and the $\frac{c}{2}$ in the left and right hand side respectively with the same constant $c$ when $X$ is a domain in $\mathbb C^d$).  The proof we give now is essentially the same as \cite[Theorem 2.1]{Bo}, but replaces the Euclidean distance function on $\mathbb C^d$ with the geodesic distance function $\rho(\cdot,x)$ on $X$, and requires some extra care to deal with the non-smoothness $\rho(\cdot,x)$ near the cut-locus of $x$.

\begin{proof}
The squared distance function $\rho^2 : X \times X \to\RR$ is continuous
and can fail to be smooth only on the cut-locus
$$B := \{
(z,z') \in  X\times X: z\mbox{ and }z' \mbox{ are conjugate points}\}
$$
of $X$.   In particular $\rho^2$ is smooth in a fixed neighbourhood
$N$ of
the diagonal.   To deal with the non-smoothness at $B$, choose a
non-negative function $\tilde{\rho}$ which is equal to $\rho$ in $N$,
smooth on $X\times X \setminus N$ and which satisfies
$$\frac{1}{2}\rho(z,z') \le \tilde{\rho}(z,z') \le 2 \rho(z,z') \text{
  for all } (z,z') \in X\times X.$$
Then $\tilde{\rho}^2$ is smooth on $X\times X$ and is equal to $\rho^2$
in $N$.

Now fix $f$ and $x\in X$, and simplify notation by writing
$\tilde{\rho}(z)$ for $\tilde{\rho}(x,z)$.  Then for
$k\gg 0$, in the region $\tilde{\rho}\le
\frac{1}{\sqrt{k}}$ we have $\tilde{\rho}(z) = \rho(z,x)$.   It will
be clear from the following argument that our constants will be
uniform in $x$ and $f$ but we leave the reader to keep track of this.

Now for a small $c>0$ set
$$ \chi_k(z) := \left\{ \begin{array}{ll} c\left( \frac{k}{2} \tilde{\rho}(z)^2 + \frac{1}{2} \right)& \text{if } \tilde{\rho}(z)\le \frac{1}{\sqrt{k}} \\ c\sqrt{k} \tilde{\rho}(z) & \text{otherwise.} \end{array} \right.$$
 One checks easily that
\begin{equation}
\frac{c\sqrt{k}}{2}\rho(z,x) -c \le \chi_k(z) \le 2c\sqrt{k}\rho(z,x) + c \text{ for all  } z\in X.\label{eq:bounddiagonalchik}
\end{equation} 
Furthermore on the region $\tilde{\rho}\le \frac{1}{\sqrt{k}}$ we have $dd^c\chi_k = O(ck)$, whereas on the complement to this region $dd^c\chi_k = O(c\sqrt{k})$.   So we may pick $c\ll 1$ small enough so that the potential
$$\zeta_k: = k\phi - \frac{1}{2}\chi_k$$
has strictly positive curvature growing at rate $O(k)$, say
\begin{equation}
dd^c\zeta_k \ge \frac{k}{2} dd^c\phi.\label{eq:diagonalcurvature}
\end{equation}
Now set
$$ v_k : = v e^{-\chi_k}$$
Since $v$ is orthogonal to the the space holomorphic sections with respect to the $L^2$-inner product defined by $e^{-k\phi}$, we see that $v_k$ is orthogonal to this space with respect to the $L^2$-inner product defined by $e^{-\zeta_k}$ (all of this is taken with respect to the volume form $\omega_{\phi}^{[d]}$ which is fixed).  Hence $v_k$ is the $L^2$-minimal solution to the equation 
$$\overline{\partial} v_k = fe^{-\chi_k} - ve^{-\chi_k} \overline{\partial} \chi_k$$
with respect to the $L^2$-inner product induced by $e^{-\zeta_k}$.  Thus by the H\"ormander estimate
\begin{align}
\int_X |v|^2_{k\phi} e^{-\chi_k} \omega_{\phi}^{[d]} &= \int_X |v_k|^2
e^{-2k\phi + \chi_k}\omega_{\phi}^{[d]}  \nonumber \\
&= \int_X |v_k|_{\zeta_k}^2 \omega_{\phi}^{[d]} \nonumber\\
&\le O(1/k) \int_X |\overline{\partial} v_k|^2_{\zeta_k}\omega_{\phi}^{[d]} \nonumber\\
&= O(1/k)\int_X |\overline{\partial}
v_k|^2_{k\phi}e^{\chi_k}\omega_{\phi}^{[d]}.
\label{e2.29.9.15}
\end{align}
We remark that the $O(1/k)$ term is bounded independent of $c$ sufficiently small by \eqref{eq:diagonalcurvature}.

Now the right hand side of \eqref{e2.29.9.15} is bounded by
\begin{equation}\label{e3.29.9.15}
O(1/k) \int_X (|f|_{k\phi}^2 e^{-\chi_k} + |v|^2_{k\phi}e^{-\chi_k}
|\overline{\partial}\chi_k|^2 ) \omega_{\phi}^{[d]} = \frac{A}{k}
\int_X |f|_{k\phi}^2 e^{-\chi_k}\omega_{\phi}^{[d]}  + Zc^2\int_X
|v|^2_{k\phi}e^{-\chi_k}  \omega_{\phi}^{[d]}
\end{equation}
where $A$ and $Z$ are uniform constants. This follows because
$\overline{\partial} \chi_k = O(c\sqrt{k})$.   Inserting
\eqref{e3.29.9.15} into \eqref{e2.29.9.15} we have
\begin{equation}
\int_X |v|^2_{k\phi} e^{-\chi_k} \omega_{\phi}^{[d]} \leq 
\frac{A}{k}
\int_X |f|_{k\phi}^2 e^{-\chi_k}\omega_{\phi}^{[d]}  + Zc^2\int_X
|v|^2_{k\phi}e^{-\chi_k}  \omega_{\phi}^{[d]}.
\end{equation}
Choosing $c$ so small that $Zc^2<1/2$, say, we 
can move the second integral to the other side, giving
$$ \int_X |v|^2_{k\phi} e^{-\chi_k} \omega_{\phi}^{[d]} \le \frac{2A}{k}\int_X |f^2|_{k\phi}e^{-\chi_k}\omega_{\phi}^{[d]}.$$
The statement of the theorem now follows from \eqref{eq:bounddiagonalchik}.
\end{proof}

\begin{proof}[Proof of Theorem \ref{thm:decaynonpartial}]
We know that $w\mapsto K_{k,y}(w)$ is holomorphic for fixed $y$, so by
the proof of Proposition~\ref{prop:Bochner},
%\begin{equation}\label{e5.29.6.14}
%\max_{w\in W} |K_{k,z}(w)|_{k\phi} \leq C_W k^d\mbox{ for }z\in W.
%\end{equation}
%From the proof of this, we know that in fact
\begin{equation}\label{e1.5.7.14}
|K_{k,x}(y)|^2_{k\phi} \leq Ck^{d}\int_{B_\delta(y)}|K_{k,x}(z)|_{k\phi}^2 \omega_{\phi,z}^{[d]}
\end{equation}
which is valid as long as $\delta$ is of order greater than or equal to $2/\sqrt{k}$.    Observe that if $\rho(x,y)\le \frac{1}{\sqrt{k}}$ then the bound we want follows from the bound $|K^{\epsilon Y}_k(x,y)|_{k\phi} \ = O(k^d)$ from \eqref{cor:pointwiseboundbergman}.  So we may assume that
$$ \rho(x,y)\ge \delta : = \frac{2}{\sqrt{k}}.$$
Fix a smooth non-negative cut-off function $\chi$ that is identically $1$ on $B_{\delta/4}(y)$ supported in $B_{\delta/2}(y)$ and such that $|\overline{\partial} \chi| =  O(\delta^{-1})$.   Observe that $\rho(x,y)\ge \delta$ implies $\chi(x)=0$.   Then we can clearly replace \eqref{e1.5.7.14} by
\begin{align}\label{e2.5.7.14}
|K_{k,x}(y)|^2_{k\phi} &= |K_{k,y}(x)|^2_{k\phi} \leq O(k^{d})\int|K_{k,x}(z)|_{k\phi}^2\chi(z)  \omega_{\phi,z}^{[d]} \\
&= O(k^d) (\chi K_{k,x}, K_{k,x})_{k\phi} = O(k^d) \cK_{k}[\chi K_{k,x}](x),
\end{align}
where, we recall, $\cK_{k}$ is the projection onto the holomorphic sections.
Now set $$f := \dbar(\chi K_{k,x}) = (\dbar \chi)K_{k,x}$$ and 
$$ v: = \chi K_{k,x} - \cK_k(\chi K_{k,x}).$$
Then
\begin{equation}\label{e4.5.7.14}
\cK_{k}[K_{k,x}\chi](x) = \chi(x)K_{k,x}(x) - v(x) = - v(x)
\end{equation}
as $\chi(x)=0$.    On the other hand, by Theorem \ref{thm:Bo}, there is a $c>0$ such that
\begin{equation}\label{e5.5.7.14}
\int |v(z)|^2_{k\phi}e^{-2c\sqrt{k}\rho(z,x)} \omega_z^{[d]} \leq O(1/k)\int
|f(z)|^2_{k\phi}e^{-\frac{c}{2}\sqrt{k}\rho(z,x)} \omega_z^{[d]}.
\end{equation}
Note that $f$ is supported in $B_{\delta/2}(y)$ and if $z\in B_{\delta/2}(y)$ then $\rho(x,y)\le \rho(z,x) + \frac{1}{\sqrt{k}}$. Thus
\begin{align}
\int |v(z)|^2_{k\phi}e^{-2c\sqrt{k}\rho(z,x)} \omega_z^{[d]} &\leq O(1/k) O(\delta^{-2}) \| K_{k,x}\|_{k\phi}^2 e^{-\frac{c}{2} \sqrt{k} \rho(x,y)} \\
& = O(k^d) e^{-\frac{c}{2} \sqrt{k} \rho(x,y)}
\end{align}
since $\|K_{k,x}\|_{k\phi}^2 = O(k^d)$ by Corollary \ref{cor:l2boundbergman}.   On the other hand by same argument with the Bochner--Martinelli formula
again, we have
\begin{equation}\label{e6.5.7.14}
|v(x)|^2_{k\phi} \leq O(k^{d})\int_{B_{\delta}(x)} |v|^2_{k\phi}\omega_{\phi}^{[d]}.
\end{equation}
Notice the quantity  $e^{-2c\sqrt{k}\rho(z,x)}$ is bounded above and below by a
constant on the ball $B_{\delta}(x)$, so this gives
\begin{align*}\label{e7.5.7.14}
  |v(x)|^2_{k\phi} &\le  O(k^{d})\int_{B_{\delta}(x)} |v|^2_{k\phi} e^{-2c\sqrt{k}\rho(z,x)} \omega_{\phi}^{[d]}\\
  &\le  O(k^{2d}) e^{-\frac{c}{2}\sqrt{k} \rho(x,y)}.
\end{align*}
Putting this with \eqref{e2.5.7.14}, \eqref{e4.5.7.14} gives the desired estimate.
\end{proof}

\subsection{Extremal Envelopes}

%Consider set of positive potentials
%$\gamma$ on $L$, as follows:
%\begin{equation}\label{e1.14.2.15}
%E:=\{ \gamma : dd^c\gamma \ge 0 \text{ and } \gamma\le \phi \text{ and } \nu_Y(\gamma)\ge \epsilon\}
%\end{equation}
%(here the notation means that $e^{-\gamma}$ is a possibly singular hermitian metric on $L$ with non-negative curvature form; and $\nu_Y(\gamma)$ denotes the Lelong number of $\gamma$ long $Y$). 

Let $Y\subset X$ be a compact complex submanifold, $L\to X$ a
holomorphic line bundle and $\phi$ a smooth strictly plurisubharmonic potential on
$L$
Following Berman 
\cite{Berman} we make the following definition:
\begin{dfn}  \label{def:extremalenvelope} Given $X,Y,L$ and $\ve\ge 0$,
the \emph{extremal envelope}, $\phi_\ve$, is defined as
\begin{equation}\label{e4.29.9.15}
\phi_{\epsilon }:=  \sup \{ \gamma \mbox{ a potential on }L: dd^c\gamma \ge 0, 
\gamma\le \phi \text{ and } \nu_Y(\gamma)\ge \epsilon\}.
\end{equation}
\end{dfn}
Here the notation means that $e^{-\gamma}$ is a possibly singular
hermitian metric on $L$ with non-negative curvature form; and
$\nu_Y(\gamma)$ denotes the Lelong number of $\gamma$ along $Y$.
Since the upper semicontinuous regularisation of $\phi_{\epsilon}$
lies in the set on the RHS of \eqref{e4.29.9.15}, it follows 
that $\phi_{\epsilon}$ is itself semicontinuous and
thus plurisubharmonic.    We will always take $\epsilon$ to be sufficiently small so that there exists such a $\gamma$, and so $\phi_\epsilon$ is not identically $-\infty$.  Then by passing to the blowup of $X$ along $Y$ one can prove easily that $\nu_Y(\phi_{\epsilon})=\epsilon$.  

\begin{dfn}
Define the \emph{forbidden region} to be the set
$$ D_{\epsilon} := \{ z\in X : \phi_{\epsilon}(z)<\phi(z)\},$$
and the \emph{equilibrium set} to be the complement
$$ X \setminus D_{\epsilon} = \{ z\in X : \phi_{\epsilon}(z) =\phi(z)\}.$$
\end{dfn}

Then clearly $D_0$ is empty, and for all $\epsilon>0$ it is a
neighbourhood of $Y$, and if $\epsilon<\epsilon'$ then
$D_{\epsilon'}\subset D_{\epsilon}$.   Again following \cite{Berman}, we now show the sets $D_{\epsilon}$ govern the PDF in that for large $k$ the function $\rho_k^\epsilon$ is exponentially small on any compact subset of $D_\epsilon$.  The following statement refines the bounds \eqref{e2.14.2.15} and
\eqref{e3.14.2.15}, using the envelope $\phi_\ve$.

% For now we consider the problem of localization of these envelopes.  If $U$ is an open neighbourhood of $Y$ then we have a set $D_{\epsilon}^U$ defined using the restricted data $\phi|_U$.

% \begin{lem}
% Let $U$ be an open neighbourhood of $Y$ in $X$ and let $\epsilon\ge 0$.  If $D_{\epsilon}^U$ is relatively compact in $U$ then $\phi^X_{\epsilon}|_U = \phi^U_{\epsilon}$ and $D_{\epsilon}^U = D_{\epsilon}^X$.
% \end{lem}
% \begin{proof}
% Let $\psi:=\phi_{\epsilon}^X|_U$.  Then $\psi\le\phi|_U$ and has Lelong number at least $\epsilon$ along $Y$ and is bounded above by $\phi|_U$, so $\psi \le \phi_\epsilon^U$.  On the other hand if $D_{\epsilon}^U$ is relatively compact in $U$ then we can extend $\phi_\epsilon^U$ to a plurisubharmonic function $\psi'$ on all of $X$ by declaring it to be equal to $\phi$ on $X\setminus U$.  As $\psi'\le \phi$ and $\nu_Y(\psi')\ge \epsilon$ we deduce $\psi'\le \phi_{\epsilon}^X$ which proves $\phi^X_{\epsilon}|_U = \phi^U_{\epsilon}$, and the final statement follows from the this.
% \end{proof}

\begin{prop}\label{prop:bergmanboundextremal}
Assume $X$ is compact.  Then there exists a constant $C$ such that
$$ \| K_{k,z}^{\epsilon}\|_{k\phi} \le Ck^{d/2} e^{-k(\phi(z) - \phi_{\epsilon}(z))} \text { for } z\in X$$
and
$$ | K_{k}^{\epsilon}(z,z')|_{k\phi} \le C k^{d} e^{-k(\phi(z) -\phi_{\epsilon}(z))} e^{-k(\phi(z') -\phi_{\epsilon}(z'))}\text{ for } z,z'\in X.$$
In particular, $K_{k,z}^{\epsilon}$ is exponentially small on $D_{\epsilon}$.
\end{prop}
\begin{proof}  Note first that since $\phi$ and $\phi_\ve$ are both
  potentials for $L$, their difference is a genuine function on $X$,
  so the right hand side of these estimates are well-defined.
Let $s\in H_{k\phi}^{\epsilon k}$.  Then by Proposition \ref{prop:Bochner} there is a constant $C$ such that
$$ |s(w)|^2_{k\phi} \le Ck^{d} \|s\|^2_{k\phi}\text{ for all } w\in X.$$
Taking the logarithm,
\begin{equation}\label{e10.14.2.15}
\frac{1}{2k} \ln |s(w)|^2  \le \phi(w) + \alpha_k\text{ for all } w\in X
\end{equation}
where 
$$\alpha_k:= \frac{\ln(Ck^{d}\|s\|^2_{k\phi})}{2k}.$$
As $s$ vanishes to order at least
$\epsilon k$ along $Y$ the plurisubharmonic potential $\ln |s|$ has
Lelong number at least $\epsilon k$ along $Y$. Combining this with \eqref{e10.14.2.15},
the potential $\frac{1}{2k}\ln |s|^2 -\alpha_k$ is a candidate for the envelope defining $\phi_{\epsilon}$ so
$$\frac{1}{2k}\ln |s(z)|^2 -\alpha_k\le \phi_{\epsilon}(z)\text{ for } z\in X.$$
So multiplying by $2k$ and then exponentiating gives
$$ |s(z)|^2_{k\phi} \le C k^d e^{-2k(\phi(z)-\phi_{\epsilon}(z)}\|s\|_{k\phi}^2 \text{ for }z\in X.$$
 Thus if $z\in X$, the square of the norm of the evaluation map
 $H_{k\phi}^{\epsilon k}\to L_z$ given by $s\mapsto s(z)$ is bounded
 by $Ck^d e^{-k(\phi(z)-\phi_{\epsilon}(z)}$ which gives the first
 statement of the Proposition.   The second statement follows by applying
 this to the section $s:= K_{k,z}^{\epsilon Y}$ to get that for
 $z,z'\in X$ 
$$ | K_{k,z}^{\epsilon}(z')|_{k\phi} \le Ck^{d/2} e^{-k(\phi(z')-\phi_\ve(z'))} \| K^\epsilon_{k,z}\|_{k\phi} \le O(k^d) e^{-k(\phi(z)-\phi_{\epsilon}(z))} e^{-k(\phi(z')-\phi_\epsilon(z'))}.$$
\end{proof} 

\begin{rem}
  Berman \cite{Berman} proves that $\phi_{\epsilon}$ is $C^{1,1}$
  although we will not use that here.   The continuity of $\phi_\ve$
  is enough to imply that for large $k$, 
$K_{k,z}^{\epsilon k}$ is exponentially small on any given compact
subset of $D_{\epsilon}.$  Berman also proves that $\rho_k^\epsilon$ is asymptotically close to $\rho_k$ on any given compact subset of $X\setminus \overline{D_{\epsilon}}$.  Thus the interest lies in the behaviour across the boundary of $D_\epsilon$ which will be the study of the rest of this paper.
\end{rem}

\section{Local Bergman kernels}

We next adapt a key idea of Berman--Berndtsson--Sj\"ostrand \cite{BBS}  who introduced the concept of a local Bergman kernel which has the reproducing property up to a term that is negligible for large $k$.  Our aim is to give a suitable definition of a local PBK that is defined on a neighbourhood of $Y$.

Again suppose that $X$ is a complex manifold of dimension $d$ and $L$ a holomorphic line bundle with given smooth strictly positive hermitian metric $e^{-\phi}$.   Also fix an open $U\subset X$ containing $Y$ and a smooth positive function $\chi$ that has compact support in $U$ and is identically equal to 1 on some open subset $W\subset U$ containing $Y$.  We recall that $\rho(z,z')$ denotes the geodesic distance between points $z,z'\in X$ with respect to a given K\"ahler metric.

\begin{dfn}\label{dfn:approxbergman}
We say a sequence of sections $B_k^\epsilon$ of $\overline{L}|_U\boxtimes L|_U$ over $U\times U$ is a \emph{local partial Bergman kernel} (local PBK) on $W$ of order $N\ge 0$ if
  \begin{enumerate}
  \item (Holomorphic) For fixed $z$ the map $z'\mapsto B^{\epsilon}_{k,z}(z') := B^\epsilon_k(z,z')$ is holomorphic.
\item (Almost Reproducing) We have
\begin{equation}\label{eq:approxbergman}
\left|f(z) - \int_U \chi(z')f(z')\overline{B}_k(z,z')e^{-2k\phi(z')}\omega_{\phi,z'}^{[d]}\right|_{k\phi}
=O(k^{-N})\|f\|_{k\phi}
\end{equation}
uniformly over $f\in H^{\epsilon k}_{k\phi}(U)$ and $z\in W$.
\item (Decay away from the Diagonal) There exist constants $C,c>0$
  such that
\begin{equation}\label{e12.14.2.15}
|B^\epsilon_{k}(z,z')|_{k\phi}  \le C k^d e^{-c\sqrt{k} \rho(z,z')}\text{ for all
}z,z'\in U.
\end{equation}
\end{enumerate}
\end{dfn}
\begin{rem}
  Our terminology differs slightly from that of
  Berman-Berndtsson-Sj\"ostrand in that we include the decay away from the
  diagonal as part of the definition.     It is sometimes convenient
  to allow $N=\infty$ in which case the $O(k^{-\infty})$ term is
  understood as meaning the bound holds for any given $N\in \mathbb
  N$.   One could relax the assumption that $B_{k,z}^\epsilon$ be holomorphic and instead assume that it is ``almost holomorphic" as in \cite{BBS}, but we will have no need for this.  When necessary we shall refer to this as a local PBK with respect to $W$ or $\chi$ if we need to emphasise the dependence on these data.
\end{rem}

The point of this definition is that, as proved in \cite{BBS}, a local
PBK approximates the globally defined PBK in a
neighbourhood of the diagonal of $W'\times W'$ for any $W'\Subset W$.

\begin{thm}(Glueing Local Partial Bergman Kernels)\label{thm:gluelocalBK}
 With notation above, suppose $B^\epsilon_k$ is a local PBK of order $N$ on $W$ and suppose that $\epsilon$ is sufficiently small so that the forbidden region $D_\epsilon$ is relatively compact in $W$.      Then if $W'\Subset W$ we have for all $r\ge 0$
 \begin{equation}
 K_{k}^{\epsilon}(x,y) = B^\epsilon_k(x,y) + O_{C^r}(k^{d+r/2-N}) \text{ for all } x,y\in W'.\label{eq:gluelocalBK}
\end{equation}

\end{thm}
\begin{proof}
The argument here is similar to that of \cite{BBS}.  Let $x,y\in W'$ and $f:=K_{k,y}^\epsilon|_U\in H^{\epsilon k}_{k\phi}(U)$.  The almost reproducing property of the local PBK $B_k^\epsilon$ gives
\begin{align}
K^{\epsilon}_k(y,x) = f(x) = \langle K_{k,y}^\epsilon, \chi B_{k,x}^\epsilon\rangle_{k\phi} + O(k^{-N}) \|f\|_{k\phi}.
\end{align}
and using Corollary \eqref{cor:l2boundbergman} this becomes
\begin{align}
  K^{\epsilon}_k(y,x) = \langle K^{\epsilon}_{k,y},\chi B_{k,x}\rangle_{k\phi} + O(k^{d/2-N}).
\end{align}
Now by definition, $\langle \chi B_{k,x}^\epsilon, K_{k,y}^\epsilon\rangle_{k\phi}$ is the value at $y$ of the $L^2$-projection of $\chi B_{k,x}^\epsilon$ onto the space of holomorphic sections that vanish to order at least $\epsilon k$ along $Y$.   That is,
$$\langle \chi B_{k,x}^\epsilon, K_{k,y}^\epsilon\rangle_{k\phi} = \chi B_{k,x}^\epsilon(y) - v(y)$$
where $v$ is the $L^2$-minimal solution of the equation 
\begin{equation}
\overline{\partial} v = \overline{\partial} (\chi B_{k,x}^\epsilon)\label{eq:dbar}
\end{equation}
among all such $v$ that vanish to order at least $\epsilon k$ along $Y$ (we remark that this makes sense as equation \eqref{eq:dbar} in particular implies that $v$ is holomorphic in a neighbourhood of $Y$ as $\chi\equiv 1$ on $W$).  Of course $v$ also depends on $x$, but we omit that from notation.

Hence we wish to bound $v(y)$ which we do with the H\"ormander technique applied with the extremal envelope $\phi_\epsilon$ from \eqref{def:extremalenvelope}.  Observe that $\phi_\epsilon$ is plurisubharmonic over all of $X$ and moreover is strictly plurisubharmonic in the support of 
$$\overline{\partial} (\chi B_{k,x}^\epsilon) = (\overline{\partial} \chi) B_{k,x}^\epsilon$$
since by hypothesis $\overline{\partial} \chi$ is supported outside of $W$ and thus within the equilibrium set $\{ z\in X: \phi_\epsilon(z) = \phi(z)\}$.    Thus we may apply the H\"ormander estimate \cite[Theorem 4.5, Chapter VIII]{Demailly} to see there exists a $v$ solving \eqref{eq:dbar} such that
$$\|v\|_{k\phi} \le \|v\|_{k\phi_\epsilon} \le C \| (\overline{\partial} \chi) B_{k,x}^\epsilon \|_{k\phi_\epsilon} =C \| (\overline{\partial} \chi) B_{k,x}^\epsilon \|_{k\phi}  $$
for some constant $C$, where the first inequality uses $\phi_\epsilon\le \phi$ and the final equality uses the statement about the support of $\overline{\partial} \chi B_{k,x}^\epsilon$.    Finally observe that as $B_{k}^\epsilon$ decays away from the diagonal we in fact have
\begin{equation}(\overline{\partial} \chi) B_{k,x}^\epsilon = O(k^{-\infty})\label{eq:useofdecaydiagonal}\end{equation}
since $x\in W'$ is a bounded distance away from the support of $\overline{\partial} \chi$.   Therefore $\|v\|_{k\phi} = O(k^{-\infty})$ which by Proposition \ref{prop:Bochner} implies $|v(y)|_{k\phi} = O(k^{-\infty})$ as well, and one sees this estimate is even uniform over $x,y\in W'$.
Thus we have proved the required statement in the $C^0$-topology (i.e.\ when $r=0$).  The statement for higher $r$ follows from this using the Cauchy-integral formula.
\end{proof}

\section{A local partial Bergman kernel for an $S^1$-invariant metric on the disc}\label{sec:disc}

In this section we construct a local Bergman kernel for a potential on
the unit disc that is circle invariant.  The method here is less
general than the construction of Berman-Berndtsson-Sj\"ostrand but is
better suited to dealing with partial Bergman kernels.  Strictly
speaking the content of this section is not needed for the proof of
our main theorem, since it will be repeated when we generalize in
Section~\ref{sec:cylinder};  we have included it here as an
illustration of the main ideas of our approach. 

\subsection{The Legendre transform}

Let  $D$ be the unit disc $\{|z|<1\}$ in $\mathbb C$.  Suppose $\phi$
is a strictly plurisubharmonic function that depends only on $|z|$
that is smooth on the closure $\overline{D}$.   We shall write
$z=e^{t+i\theta}$ and $\phi=\phi(t)$.  Then consider the Legendre
transform 
\begin{equation}\label{e11.7.2.13}
u(x) + \phi(t) = xt,\;\; x = \phi'(t), t = u'(x).
\end{equation}
This is well defined since the assumption that the potential is
strictly plurisubharmonic means $\phi$ is a strictly convex function
of $t$ and hence $u$ is also strictly convex.   We shall refer to
$x=:\mu(z)$ as the \emph{dual variable} (or \emph{momentum variable})
and to $u(x)$ as the \emph{symplectic potential}.     This $x$ is
defined up 
to an additive constant, which we choose so $x=0$ corresponds to
$z=0$ and suppose that $|z|=1$ corresponds to $x = a>0$ (the reader
may wish to recall the model case in which $\phi = |z|^2 = e^{2t}$ has
$u(x) = \frac{1}{2} (x\ln x - x)$ as the symplectic potential).  Any function of $|z|$ defined on an annulus $\{ \alpha \le |z|\le \beta\}$ can then be thought of as a function of $x$ and $\theta$, where $x$ ranges in some interval and $\theta$ ranges in $[0,2\pi]$, and will do so henceforth without further comment.  Moreover, the volume form $dd^c\phi$ on this annulus becomes the standard measure $dxd\theta$ in the variables $x,\theta$.
\medskip 

Now define
$$\nu : = \frac{n}{k},$$
and let $s_{n,k}$ be the monomial
\begin{equation}\label{e12.7.2.13}
s_{n,k}(z) := e^{-ku(\nu)}z^n = e^{-ku(\nu)}e^{n(t+i \theta)} =
e^{-k(u(\nu) - \nu(t+i\theta))}.
\end{equation}
Then
\begin{equation}\label{e13.7.2.13}
|s_{n,k}(z)|_{k\phi} = e^{-k(\phi(t) + u(\nu) - \nu t)}
\end{equation}
so if we define
\begin{equation}\label{e14.7.2.13}
U(\nu,x) := u(\nu) - u(x) + (x-\nu)u'(x)
\end{equation} 
and replace $\phi(t)$ and $t$ by the dual variables $u(x)$ and $x$ in
\eqref{e13.7.2.13} we have
\begin{equation}
|s_{n,k}(z)|_{k\phi}  = e^{- kU(\nu,x)}.
\end{equation}
In particular, the normalization of $s_{n,k}$ is such that the maximum
value of $|s_{n,k}(z)|_{k\phi}$ is $1$ and is attained for those $z$ for
which $x = \nu$ (this follows easily from the convexity of $u$).  In fact,
$$
U_x(\nu,x) = (x-\nu)u''(x)
$$
so the unique turning point of $x\mapsto U(\nu,x)$ is at $x=\nu$.

\subsection{The Local Partial Bergman Kernel}
Fix $\sigma\in (0,a)$ and let $\chi=\chi(x)$ be a smooth positive cutoff function that is identically equal to 1 for $x\le \sigma$ and has compact support in $D$.  For $\epsilon<\sigma$ define
\begin{equation}\label{def:akepsilon}
A^{\epsilon}_{k}(z,z') := \sum_{\epsilon\le \nu \leq \sigma}
G_{n,k} \overline{s_{n,k}(z)}s_{n,k}(z') \text{ for } (z,z')\in D\times D
\end{equation}
where
\begin{equation}\label{e3.10.2.13}
\frac{1}{G_{n,k}}:= \int_D \chi(x) e^{-2kU(\nu,x)}\,\rd x.
\end{equation}
Finally let $Y = \{0\}$ be the origin in $D$.

\begin{thm}  
\label{t1.14.2.13}
Fix $\eta<\sigma$.   Then  $A^{\epsilon}_k$ has the almost reproducing property on $\mu^{-1}[0,\eta)$; that is
\begin{equation}\label{e2.11.2.13}
\left| \int_D \chi(x')f(z') \overline{A^\epsilon_{k}(z,z')}e^{-2k\phi(t')}\,\rd x'\rd\theta' -
f(z)\right|_{k\phi}
= O(k^{-\infty})\|f\|_{k\phi}
\end{equation}
for all $f\in H_{k\phi}^{\epsilon k|}(D)$ and $z\in \mu^{-1}[0,\eta)$.
\end{thm}
\begin{proof}
Let $f\in H_{k\phi}^{\epsilon k|}(D)$ which can be written as  $f = f_1 + f_2$ where 
\begin{equation}\label{e5.11.2.13}
f_1(z) = \sum_{\epsilon \le \nu \leq \sigma } a_n z^n \text{ and } f_2(z) = \sum_{\nu > \sigma } a_n z^n
\end{equation}
and $\nu=n/k$ as before.  Obviously $f_1$ and $f_2$ are $L^2$-orthogonal, so
\begin{equation}\label{e6.11.2.13}
\|f\|_{k\phi}^2 = \| f_1\|^2_{k\phi} + \| f_2 \|_{k\phi}^2.
\end{equation}
Now by symmetry,
\begin{equation}\label{e7.11.2.13}
\int_D \chi(z')f(z') \overline{A_k^{\epsilon}(z,z')}e^{-2k\phi(t')} dz'= f_1(z),\;\;
\end{equation}
%We will deal with $f_1$ and $f_2$ separately in \eqref{e2.11.2.13}.
so all that we have to do is prove
$$ |f_2(z)|_{k\phi} \le O(k^{-\infty}) \|f_2\|_{k\phi} \text{ for } x=\mu(z)<\eta.$$

To this end, fix some $\tau$ with $\phi'(\tau)\in(\eta,\sigma)$.  If $z=e^{t+i\theta}$ has $\mu(z)<\eta$ then $\phi'(t)<\eta<\phi'(\tau)$ and so by strict convexity of $\phi$ this implies $t\le \tau-\delta$ for some $\delta>0$ uniformly over all such $z$.   Then Cauchy's inequalities give
\begin{equation}\label{e3.12.2.13}
|a_ne^{n\tau}| \leq \sup_{\theta} |f_2(e^{\tau+i\theta})| =
e^{k\phi(\tau)}
\sup_{\theta} |f_2(e^{\tau+i\theta})|_{k\phi}.
\end{equation}
Therefore
\begin{equation}\label{e5.12.2.13}
|f_2(z)|_{k\phi} \leq \sum_{n>\sigma k} |a_nz^n|e^{-k\phi(t)}
\leq e^{-k\phi(t)} \sup_{\theta}|f(e^{\tau +i\theta})| \sum_{n >
  \sigma k} e^{n(t-\tau)}.
\end{equation}
Now, by Proposition~\ref{prop:Bochner}, there is a constant $C$ such that if $t<\tau$ then
\begin{equation}\label{e4.12.2.13}
|f_2(e^{t+i\theta})|_{k\phi} \leq Ck^{\frac{1}{2}}\|f_2\|_{k\phi},
\end{equation}
and hence
\begin{equation}\label{e6.12.2.13}
|f_2(e^{t+i\theta})|_{k\phi} \leq Ck^{\frac{1}{2}} \|f_2\|_{k\phi}\,e^{-k[\phi(t) -
  \phi(\tau)]}e^{k\sigma[t-\tau]}.
\end{equation}

Thus we will get an exponentially small multiple of $\|f_2\|_{k\phi}$
provided that 
\begin{equation}\label{e7.12.2.13}
\phi(t) - \phi(\tau) - \sigma (t-\tau) >\delta'
\end{equation}
for some positive $\delta'$.    But this is clear from the following picture:
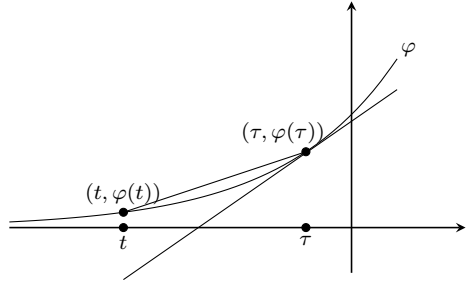
\begin{center}
\begin{tikzpicture}[>=stealth]
%
%  t-line  
%
{\scriptsize
\begin{scope}[scale = 3]
\draw[semithick, ->] (-1.5,0) -- (.5,0);
\draw[semithick, ->] (0,-.2) -- (0,1);
%
% Label axes
%
%\draw (-1.7,0) node [below] {$t$};
%\draw (-1.8,.8) node {Red line: $\varphi(t)$};
\draw[domain=-1.5:0.2,smooth]	plot (\x,{0.5*exp(2*\x)});
%\draw (-1,1) node {$\phi(t) = e^{2t}/2$};
%
%   The key points
%
%\draw (-1.8,.75) node[below] {$x = \phi_t$};
\filldraw [color = black] (-1, {0.5*exp(-2)}) circle (.5pt) node [above,color=black]{$(t,\phi(t))$};
\filldraw [color = black] (-.2, {0.5*exp(-.4)}) circle (.5pt);
\draw (-.3, {0.5*exp(-.4)}) node [above]
{$(\tau,\phi(\tau))$};
\filldraw [color = black] (-.2, 0) circle (.5pt)
node[below,color=black] {$\tau$};
\filldraw [color = black] (-1, 0) circle (.5pt)
node[below,color=black] {$t$};

\filldraw [color = black] (0.25, 0.85) 
node[below,color=black] {$\phi$};

\draw [black] (-1, {0.5*exp(-2)}) --
(-.2, {0.5*exp(-.4)})  ;
\draw[color=black,domain=-1:0.2,smooth]	plot (\x, {.7*\x+.47});

\end{scope}
% %
% %  x-line
% %
% \begin{scope}[yshift = -6cm,xshift=-5cm]
% \draw[semithick, ->] (-.2,0) -- (5,0);
% \draw[semithick, ->] (0,-.2) -- (0,5);
% %
% % Label axes
% %
% \draw (4,0) node [below] {$x$};
% \draw (0,4) node [right] {$u(x)$};
% \draw (3,2) node {$t = u_x$};
% \draw[color=red,domain=0.01:4,smooth]	plot (\x, {.5*\x*ln(\x) - .5*\x});
% \filldraw (1.5,0) circle (1pt) node [below] {$\eta$};
% \filldraw (2.5,0) circle (1pt) node [below] {$\sigma$};
% \filldraw (2,0) circle (1pt) node [below] {\scriptsize $\phi_t(\tau)$};
% \end{scope}
}
\end{tikzpicture}
\end{center}
That is, by convexity of $\phi$ the quotient
\begin{equation}\label{e8.12.2.13}
\frac{\phi(\tau) - \phi(t)}{\tau - t}
\end{equation}
is bounded above by the
slope $\phi'(\tau)$ which assumed to be strictly less than $\sigma$  as $t\le \tau-\delta$.
\end{proof}

\section{A local $S^1$-invariant Bergman kernel on $X$}
\label{sec:cylinder}
\subsection{Outline of the construction} We return to the case in
which $Y$ is a divisor in a compact complex manifold $X$, and $L$ is a
complex line bundle with hermitian metric $e^{-\phi}$ that has
strictly positive curvature. We assume from now on that there exists a
neighbourhood $U$ of $Y$ that admits a holomorphic $S^1$-action so
that all the data are invariant when restricted to $U$ (that is, the
action lifts to $L|_U$ preserving the hermitian metric, and $Y$ is
fixed pointwise by the action).  In this section we shall construct a
local PBK on $U$ whose asymptotics we can understand as $k$ tends to
infinity. 

There are several ingredients to this construction.  First we exploit the relation between the eigenspaces of the $S^1$-action on $H^0(U,L^k)$ and the order of vanishing along $Y$.  By hypothesis, $S^1$-acts on the fibre $L_p$ for any $p\in Y$ with some weight $w$ (which is the same for every $p\in Y$) and by renormalising the action we may assume without loss of generality that $w=0$.  Write
$$ H^0(U,\mathcal L^k) = \bigoplus_n V_k(n)$$
where $V_k(n)$ is the subspace of elements of weight $n$.  Then as $Y$ is fixed pointwise
$$ H^0(Y, L^k\otimes \mathcal I_Y^n) = \bigoplus_{j\ge n} V_k(j).$$
Using the presence of the $S^1$-action, we then have an identification
\begin{equation}
 V_k(n) = H^0(U,L^k\otimes \mathcal I_Y^{n+1}) / H^0(U,L^k\otimes \mathcal I_Y^n) = H^0(Y,L^k\otimes \mathcal O(-nY)|_Y).\label{eq:Vkn}
\end{equation}
This is made explicit by a choice $\sigma\in H^0(U,\mathcal O(Y))$ of defining section for $Y$.  Thinking of $U$ as a disc bundle $\pi\colon U\to  Y$, the isomorphism \eqref{eq:Vkn} is given by pulling back to $U$ and then multiplying by $\sigma^n$.  To put it another way, if $f$ is a holomorphic section of $L^k|_U$ we have an expansion
$$ f = \sum_n f_n \sigma^n$$
where $f_n \in H^0(Y, L(-\nu Y)|_Y^k)$ and $\nu: = n/k$.  When $U$ is
a disc centered at the origin with the standard $S^1$-action and $Y=
\{0\}$, this just reduces to power series expansion of holomorphic
functions on the disc. 

Next we define a hermitian inner product on $H^0(Y, L(-\nu Y)^k|_Y)$
as follows.  Fix an $S^1$-invariant cut-off function $\chi$ which is
identically 1 in a neighbourhood of $Y$ and supported in $U$.  Then,
with $\sigma$ as above, we set 
\begin{equation}\label{e11.29.9.15}
\|f_n\|_{\nu,k,\chi}^2 = \int_X \chi \pi_*(|f_n|^2|\sigma|^{2n} )
\omega^{[d]} \text{ for } f_n \in H^0(Y, L(-\nu Y)^k|_Y).
\end{equation}
It turns out that this inner product is equal to the $L^2$-inner-product with respect to certain hermitian metrics $e^{-2\eta_\nu}$ on $L(-\nu Y)^k|_Y$ and certain volume forms on $Y$:

\begin{prop} (Proposition \ref{prop:l2})
Fix $\epsilon'$ slightly larger than $\epsilon$.  Then for $\nu\in [0,\epsilon']$ there is a smooth  strictly plurisubharmonic potential $\eta_{\nu}$ on $L(-\nu Y)|_Y$ and a volume element $dV_{\nu,\hbar}$ on $Y$ depending smoothly on $\nu$ and $\hbar := 1/\sqrt{k}$, such that
$$ \| f_n\|^2_{\nu, k,\chi} = \int_X |f_n|^2_{k\eta_\nu} dV_{\nu,\hbar}.$$
\end{prop}

The proof of this proposition is essentially a standard stationary
phase argument for the integral \eqref{e11.29.9.15}.   To give a rough
idea of the argument, consider the local situation.  In 
standard local coordinates $(z=e^{t+i\theta},w_1,\ldots,w_n)$ so that
$Y$ is given locally by $z=0$, our potential $\phi = \phi(t,w)$ and is
convex in $t$ for fixed $w$. Introduce (locally) the 
parametrized Legendre transform $u$ characterised by
$$ u(x,w) + \phi(t,w) = xt \text{ where } x= \phi_t \text{ and } t =
u_x.$$
Here $x$ is the moment map of the $S^1$-action.  On the other hand, 
the integrand in \eqref{e11.29.9.15} then takes the form
$$
|f_n(w)|^2\exp(2k(u(x,w) - (x-\nu)u_x(x,w)))\,\frac{\rd x \wedge \rd\theta}{2\pi}
$$
multiplied by a smooth volume element in the $w$ variables, where $\nu
= n/k$.  So standard asymptotic methods give a leading term of the
form
$$
\frac{1}{\sqrt{u_{xx}(\nu,w)}k}
e^{2ku(\nu,w)}|f_n(w)|^2
$$
on integrating with respect to $x$ and $\theta$, since the exponential
is stationary at $x=\nu$.  Thus we see the negative of the Legendre
transform $-u(\nu,w)$ appear as a local potential for $L(-\nu Y)|_Y$.
We shall show that $w \mapsto - u(\nu,w)$ is strictly plurisubharmonic in $w$ and
that these local potentials indeed
patch together to give a potential $\eta_\nu$ on $L(-\nu Y)|_Y$.  The
existence, and properties, of the required volume  $dV_{\nu,\hbar}$
then comes from an application of Laplace's method. 

Now the projection from the $L^2$-sections of $L(-\nu Y)|_Y^k$ to the holomorphic ones is given by an integral kernel on $Y\times Y$ that we denote by $G_{n,k}$.  Putting all of this together we shall prove:

\begin{thm}(Theorem \ref{thm:constructionlocalPBK}).  Fix $\epsilon'$ slightly larger than $\epsilon$.  Then the quantity
$$ B_k^\epsilon : = \sum_{n= \epsilon k}^{\epsilon 'k } G_{n,k} \sigma^n\boxtimes \overline{\sigma}^n$$
is an $S^1$-local PBK for $Y$.
\end{thm}
We refer the reader to Section \ref{sec:modified} for the precise definition of a $S^1$-local PBK (which merely modifies slightly the decay away from the diagonal property).  Of course this theorem is to be understood as holding with respect to
the chosen cut-off function $\chi$.  This is a rather precise formula
for the local PBK, which we shall exploit to understand its
asymptotics as $k$ tends to infinity.

We now go through the details of these two results.

\subsection{Circle-invariant set-up} \label{sec:circlesetup} et $X,Y,L$ and $\phi$ be as before and set $\omega = dd^c\phi$.  We also fix a defining section $\sigma\in H^0(Y,\mathcal O(Y))$ for $Y$.  We suppose that there is an open neighbourhood $U$ of $Y$ that admits a holomorphic $S^1$-action with $Y$ as the fixed point set, so that $U$ is covered by charts that admit ``standard coordinates" $(z,w)$ so that
\begin{equation}
\lambda\cdot (z,w) = (\lambda z,w)\label{eq:locals1coordinates}
\end{equation}
and so that $\sigma=z$ in these coordinates.  We also assume that this actions lifts to $L|_U$ preserving $\phi$ in such a way that the $S^1$-action is trivial over points in $Y$.  By abuse of notation we let $\mathcal O(Y)$ denote both the line bundle associated to the divisor $Y$ as well as its sheaf of sections (so is equal to the normal bundle of $Y$ in $X$).  We let $\mu\colon X\to \mathbb R$ denote the Hamiltonian of this $S^1$-action normalised so $Y = \mu^{-1}(0)$ and always assume that $\epsilon$ is sufficiently small so that $\mu^{-1}[0,\epsilon)\Subset U$.

As we have a holomorphic $S^1$-action, we can identify our tubular neighbourhood $U$ of $Y$ holomorphically with a disc-subbundle of $\mathcal O(Y)$.

\begin{lem}
The neighbourhood $U$ of $Y$ is biholomorphic to a disc subbundle $\pi\colon D\to Y$ of $\mathcal O(Y)$.  Furthermore any $S^1$-invariant line bundle $L'$ on $U$ is canonically isomorphic to $\pi^* L'|_Y$.
\end{lem}
\begin{proof} Let $(z_\alpha,w_\alpha)$ and $(z_\beta,w_\beta)$ be two sets of standard coordinates.   The transition functions between them are necessarily of the form
\begin{equation}
z_\beta = \lambda_{\alpha\beta}(z_\alpha,w_\alpha)z_\alpha\text{ and } w_\beta = \tau_{\alpha\beta}(z_\alpha,w_\alpha)
\end{equation}
where $\lambda_{\alpha\beta}$ is holomorphic and takes values in
$\bC^*$.  Then both $\lambda_{\alpha\beta}$ and $\tau_{\alpha\beta}$ must
be $S^1$-invariant as well, which means their dependence upon
$z_\alpha$ must be trivial.  Hence we have
\begin{equation}
z_\beta = \lambda_{\alpha\beta}(w_\alpha)z_\alpha\label{eq:zbetazalpha}\end{equation}
is actually linear in the $z$ coordinate.  Since locally $Y$ is given by $z_\alpha=0$ (resp.\ $z_{\beta}=0$) the $\lambda_{\alpha\beta}$ are the transition functions for the line bundle $\mathcal O(Y)$, so we have the desired biholomorphism.  The argument for the second statement is the same, as the transition functions for an $S^1$-invariant bundle $L'$ must be independent of the normal variables $z_{\alpha}$ in standard coordinates.
\end{proof}

Now we recall the eigenspace decomposition we have denoted by
$$ H^0(U,L^k) = \bigoplus_n V_k(n)$$
where $V_k(n)$ denotes the subspace of weight $n$ and that here and henceforth we are setting 
$$ \nu: = \frac{n}{k}.$$

\begin{lem}  The map
$$ H^0(Y, L(-\nu Y)|_Y^k) \to V_k(n) \text{ given by } f_n \mapsto \pi^* f_n\sigma^n$$
is an isomorphism.
\end{lem}

We omit the proof.

\begin{lem}\label{lem:expansion} Given $f\in H^0(U,L^k)$ there is a unique sequence
$$ f_n \in H^0(Y,L(-\nu Y)|_Y^k), \text{ for } n\ge 0$$
such that
$$ f= \sum_{n\ge 0} (\pi^* f_n) \sigma^n.$$
\end{lem}
\begin{proof}
We have that $(\pi^* f_n) \sigma^n$ is just the weight-$n$ component of $f$ with respect to the $S^1$-action.  We use here the fact that a weight-zero section on $U$ is canonically the same as the pull-back of a section over $Y$.  That $f_n$ lives where claimed is immediate from the definitions.
\end{proof}

From now on we regard $U$ as a disc-subbundle of $\mathcal O(Y)$ without further comment, and drop the $\pi^*$ from notation so we identify $L^k\otimes \mathcal O(-nY)$ with $\pi^*(L(-\nu Y)^k|_Y)$ and also identify sections $f_n$ of $L(-\nu Y)|_Y^k)$ with their pull-back $\pi^* f_n$ to $U$. \\

Our next order of business is to define hermitian inner products on $H^0(Y, L(-\nu Y)|_Y^k)$.  Pick $\epsilon'$ slightly larger than $\epsilon$ and a domain $\Omega'$ arranged as follows:

$$\mu^{-1}[0,\epsilon) \Subset \Omega' \Subset \mu^{-1}[0,\epsilon')\Subset U.$$

\begin{dfn}\label{def:innerproduct}
Fix once and for all an $S^1$-invariant cut-off function
$$ \chi \in C_0^{\infty}(U) \text{ with } \chi\equiv 1 \text{ on } \mu^{-1}[0,\epsilon'].$$
Then for $f_n,g_n\in C^{\infty}(Y, L(-\nu Y)|_Y^k)$ define
\begin{equation}\langle f_n, g_n\rangle _{\nu,k,\chi}:= \int_U \chi (f_n, g_n)_{k\phi} |\sigma|^{2\nu k} \omega^{[d]}.\label{def:innerproducts}
\end{equation}
(We emphasise our abuse of notation in that on the right hand side $f_n$ and $g_n$ are being identified with their pullback, so as to be defined on $U$).
\end{dfn}

So by construction if $f=\sum_n f_n \sigma^n$ and $g= \sum_n g_n \sigma^n$ are the expansions of two functions in $H^0(U,L^k)$ as in Lemma \ref{lem:expansion} we have
$$\int_X \chi (f,g)_{k\phi} \omega^{[d]}  = \int_U \chi (f,g)_{k\phi} \omega^{[d]} =  \sum_n \langle f_n,g_n\rangle _{\nu,k,\chi}.$$

\subsection{The local Legendre Transformation} Our next goal is 
better 
to understand the hermitian inner products on $H^0(Y, L(-\nu Y)|_Y^k)$
defined in \eqref{def:innerproducts}.  In this section we shall do so
locally over a small chart in $Y$, and in the next we will see how
this globalises over $Y$.  So assume we have standard coordinates $z,
w_1,\ldots,w_{d-1}$ as in \eqref{eq:locals1coordinates} on some patch
of the form $U_{\alpha} = \pi^{-1}(W_\alpha)$ for some open subset
$W_{\alpha}$ of $Y$ and write 
$$z:= e^{t+i\theta}.$$
 Recall that this includes the assumption that the defining section $\sigma$ for $Y$ is locally given by $\sigma=z$ in this coordinate system, and on this chart the hermitian metric on $L$ is given by $e^{-2\phi}$ where $\phi$ is a function of $t$ and $w$.  So, for $f$ supported in such a coordinate chart
$$ \|f\|^2_{\nu,k,\chi} = \int \chi e^{-2k(\phi-\nu t)} |f|^2 \omega^{[d]}.$$
We fix $\nu>0$ and investigate the large $k$ asymptotic behaviour of this integral.  By standard principles doing the $t$-integral first, this will be exponentially small in $k$ unless the cricial point $\phi_t=\nu$ occurs in the support of $f$.  Assuming this to be the case, the main term for large $k$ is
$$ e^{-2k(\phi(t_\nu) - \nu t_\nu)} = e^{2ku(\nu,w)}$$
where $u$ is the parametrized Legendre transform defined as follows.  Introduce the dual variables
\begin{equation} x: = \phi_t\label{eq:xequalsphit}\end{equation}
which is precisely the moment map of the $S^1$-action, so $Y$ is given by $x=0$.

\begin{dfn}
The parameterized Legendre transform $u$ is characterized by 
\begin{equation}
\phi(t,w) + u(x,w) = tx.\label{eq:parameterizedlegendretransform}
\end{equation}
\end{dfn}
Thus dually to \eqref{eq:xequalsphit} we have
$$ t= u_x.$$

\begin{lem}\label{lem:kiselman}
For each fixed $x$ the function $w\mapsto -u(x,w)$ is strictly plurisubharmonic.   Moreover, if we
introduce the connection $1$-form
\begin{equation}\label{e4.23.2.13}
\alpha = \frac{1}{\phi_{tt}}\rd ^c x = 
\frac{1}{2\pi\phi_{tt}}J\rd x =\frac{1}{2\pi}\left(\rd \theta +
\frac{i}{\phi_{tt}}\left( \phi_{t\overline{a}}\rd
    \overline{w}^{\overline{a}}  - \phi_{ta}\rd w^a\right)\right),
\end{equation}
then
\begin{equation}\label{e5.23.2.13}
\omega = \omega_\phi = \rd x \wedge \alpha + (\rd  \rd^c)_w( - u(x,\cdot)),
\end{equation}
where the notation indicates fixing $x$ and computing $\rd \rd^c$ of $w
\mapsto  - u(x,w)$. 

\end{lem}
\begin{proof}
By elementary computation,
\begin{eqnarray*}
\rd \phi &= &\phi_t\,\rd t + \phi_{a}\rd w^a + \phi_{\overline{a}}\rd
\overline{w}^{\overline{a}}, \\
J\rd \phi &= &\phi_t\,\rd \theta - i \phi_{a}\rd w^a + i\phi_{\overline{a}}\rd
\overline{w}^{\overline{a}}, \\
\rd J \rd \phi &=& \phi_{tt}\rd t\wedge \rd \theta +
\phi_{ta}\rd w^a\wedge \rd\theta  + \phi_{t\overline{a}}\rd\overline{w}^{\overline{a}}\wedge \rd\theta  \\
& + &  i(\phi_{t\overline{a}}\rd t\wedge d\overline{w}^{\overline{a}} - \phi_{ta}\rd t\wedge dw^a)
+2i\phi_{a\overline{b}}\rd w^a \wedge \rd\overline{w}^{\overline{b}}.
\end{eqnarray*}
Since $x = \phi_t$,
\begin{equation}\label{e6.23.2.13}
\rd x = \phi_{tt}\,\rd t + \phi_{ta}\rd w^a + \phi_{t\overline{a}}\rd w^{\overline{a}}
\end{equation}
so
\begin{eqnarray}\label{e7.23.2.13}
\rd x \wedge \alpha &=&
\frac{1}{2\pi}\left\{\phi_{tt}\,\rd t + \phi_{ta}\rd w^a + \phi_{t\overline{a}}\rd
w^{\overline{a}}\right\}
\wedge
\left\{\rd \theta +
\frac{i}{\phi_{tt}}( \phi_{t\overline{a}} \rd
    \overline{w}^{\overline{a}}  - \phi_{ta}\rd w^a)
\right\} \\
&=&
\rd \rd^c \phi +\frac{i}{\pi}
\left\{\frac{1}{\phi_{tt}}\phi_{ta}\phi_{t\overline{b}}
   - \phi_{a\overline{b}}\right\}\rd  w^a \wedge
\rd   \overline{w}^{\overline{b}}.
\end{eqnarray}
Hence we need to prove that
\begin{equation}\label{e8.23.2.13}
(\rd J\rd)_w u = 
2i\left\{\frac{1}{\phi_{tt}}\phi_{ta}\phi_{t\overline{b}}
   - \phi_{a\overline{b}}\right\}\rd  w^a \wedge
\rd   \overline{w}^{\overline{b}}.
\end{equation}
This follows by careful differentiation of \eqref{eq:parameterizedlegendretransform}. First,
by differentiation with respect to $w^a$,
\begin{equation}\label{e9.2.23.13}
\phi_a(t,w) + u_a(x,w) = 0
\end{equation}
where in the first term we are holding $t$ fixed and in the second we
are holding $x$ fixed.  Differentiating again with respect to
$\overline{w}^b$, holding $x$ fixed,
\begin{equation}\label{e10.2.23.13}
\phi_{a\overline{b}} + \phi_{ta}t_{\overline{b}} = -
u_{a\overline{b}}
\end{equation}
Next, differentiation of $x = \phi_t(x,w)$ with respect to yields
$\overline{w}^b$ (holding $x$ fixed)
\begin{equation}\label{e11.2.23.13}
\phi_{t\overline{b}} + \phi_{tt}t_{\overline{b}} = 0.
\end{equation}
If we insert this into \eqref{e10.2.23.13}, we obtain
\begin{equation}\label{e12.2.23.13}
- u_{a\overline{b}} =
\phi_{a\overline{b}} - \phi_{ta}\phi_{t\overline{b}}/\phi_{tt}
\end{equation}
This proves \eqref{e8.23.2.13}, from which it follows that $-u(x,w)$ is strictly plurisubharmonic in $w$ for fixed $x$.
\end{proof}

\begin{rem}
The statement that for fixed $x$ the map $w\mapsto -u(x,w)$ is strictly plurisubharmonic follows also from the Kiselman minimum principle \cite{Kiselman}.  The advantage of the above calculation is that it also gives an explicit formula for its curvature.
\end{rem}

\begin{lem}\label{lem:ustrictlyconvex} For fixed $w$ the map $x\mapsto u(x,w)$ is strictly convex.
\end{lem}
\begin{proof}
This follows as $u_{xx} = t_x$ and $x_t= \phi_{tt}>0$ as $\phi$ is assumed to have strictly positive curvature.
\end{proof}

\subsection{The Global Legendre Transform}

\begin{dfn} Let $u(x,w)$ be the locally defined Legendre transform from \eqref{eq:parameterizedlegendretransform}.   We set
$$\eta_\nu(w) : = -u(\nu,w)$$
\end{dfn}

\begin{lem}\label{lem:etanugloba}
The above locally defined expression for $\eta_\nu$ gives a well-defined potential on $L(-\nu Y)|_Y$ whose curvature $dd^c\eta_\nu$ is strictly positive (and bounded away from $0$ as $\nu$ ranges in a bounded interval).
\end{lem}
\begin{proof}
Consider a cover by standard coordinates $(z_\alpha,w_\alpha)$ on charts $U_\alpha$ as in \eqref{eq:locals1coordinates}.  We assume the line bundle $L$ is trivialized over each $U_\alpha$ with transition functions $\mu_{\alpha\beta}$ and that the metric on $L$ has potential $\phi_\alpha$ over $U_\alpha$.  As we have already seen \eqref{eq:zbetazalpha}
$$ z_\beta = \lambda_{\alpha\beta} z_\beta$$
where $\lambda_{\alpha\beta}$ is a holomorphic function of $w_{\alpha}$ in $U_{\alpha}\cap U_{\beta}$ giving the transition functions for $\mathcal O(Y)|_Y$.
Now $\partial_{t_\alpha} \phi_\alpha = \partial_{t_\beta} \phi_\beta$ since they are both equal to the globally defined moment map $x$.  So by definition of the Legendre transform along the set $\{x=\nu\}$ we have
\begin{align*}
 \phi_\alpha(t_\alpha,w_\alpha) + u_\alpha(\nu,w_\alpha) = t_\alpha \nu \\
 \phi_\beta(t_\beta,w_\beta) + u_\beta(\nu,w_\beta) = t_\beta \nu 
 \end{align*}
Now
$$t_\alpha = \log |z_\alpha| = \log |\lambda_{\alpha\beta}| + \log |z_\beta| = \log |\lambda_{\alpha\beta}| + t_\beta$$
and as $\phi$ is a potential on $L$ we have $\phi_{\alpha} = \phi_\beta + \log |\mu_{\alpha\beta}|$.  So we obtain
$$ u_\beta(\nu,w_\beta) = t_\beta \nu - \phi_\beta = t_\alpha \nu - \phi_\alpha + \log (|\mu_{\alpha\beta} |) -\nu \log(|\lambda_{\alpha\beta}|) = u_{\alpha}(\nu,w_\alpha) + \log (|\mu_{\alpha\beta}| |\lambda_{\alpha\beta}|^{-\nu}) $$
which is precisely the statement that $\eta_\nu(\cdot) = -u(\nu,\cdot)$ is a well-defined potential on $L(-\nu Y)|_Y$.  The positivity of the curvature of this potential is a local calculation, and is the content of Lemma \ref{lem:kiselman}
\end{proof}

One of our main uses of this potential is the following expression for
the hermitian metric we defined on $H^0(Y,L(-\nu Y)|_Y^k)$.  Write
$\omega_{x,u} = -(\rd\rd^c)_w u(x,w)$, 
so by Lemma \ref{lem:kiselman}
$$ \omega_{\phi} = dx \wedge \alpha + \omega_{x,u},\mbox{ with }
\alpha= \frac{1}{\phi_{tt}} \rd^cx
$$
as before.

\begin{prop}\label{prop:l2}
There exist volume forms $dV_{\nu,\hbar}$ on $Y$ for $\nu\in
(0,\epsilon')$ such that the inner-product on sections of $L(-\nu
Y)^k_Y$ defined in \eqref{def:innerproduct} is the $L^2$-inner-product
with respect to the potential $\eta_\nu$ and volume form
$dV_{\nu,\hbar}$, i.e. 
$$ \|f\|_{\nu,k,\chi}^2 = \int_Y |f|^2_{\eta_\nu} dV_{\nu,\hbar} \text{ for all } f\in C^{\infty}(Y,L(-\nu Y)|_Y^k).$$
In fact
\begin{equation}
dV_{\nu,\hbar} = \hbar \sqrt{\frac{2\pi}{u_{xx}(\nu,w)}} \omega_{\nu,u}^{[d-1]}A(\nu,\hbar)\label{eq:dV}
\end{equation}
where $A(\nu,\hbar)$ is a smooth function and $A(\nu,0)=1$.
\end{prop}
\begin{proof}
Let $(z,w_1,\ldots,w_{d-1})$ be standard coordinates on some chart of the form $U_\alpha = \pi^{-1}(W_\alpha)$ for some $W_{\alpha}\subset Y$.  Without loss of generality we may assume that $f$ is supported in $W$.   Observe that 
$$\omega_{\phi}^{[d]} = dx\wedge \alpha \wedge \omega_{x,\nu}^{[d-1]}
= \frac{1}{2\pi} dx\wedge d\theta \wedge \omega_{x,u}^{[d-1]}$$ 
Then chasing definitions
$$  \|f\|_{\nu,k,\chi}^2 = \frac{1}{2\pi}\int _{U_\alpha} \chi
|f(w)|_{\eta_\nu}^2 e^{-2k(\phi - \nu t - \eta_\nu)} dx\wedge
d\theta\wedge \omega_{x,u}^{[d-1]}.$$   
This can be calculated by performing the $x$-integral first, to obtain an expression as an integral on $W_{\alpha}$.  This gives the existence of the volume form $dV_{\nu,\hbar}$ (which is clearly well-defined over all of $Y$).    Now observe that
\begin{align*}
 \phi(t,w) - \nu t - \eta_{\nu}(w)&= \phi(t,w) -xt - (\nu-x) t + u(\nu,w) \\
 &= - u(x,w) - (\nu-x) u_x(x,w) + u(\nu,w) \\
 &=  u_{xx}(x,w)(x-\nu)^2 + q(x,\nu,w)
\end{align*}
where $q$ vanishes to order at least 3 at $x=\nu$.  Thus the  $x$-integral can be calculated using Laplace's method, proving that $dV_{\nu,\hbar}$ is smooth in $\hbar$, and moreover giving the stated leading term for $dV_{\nu,\hbar}$ (this is a simple case of the Laplace method described in Appendix \ref{appendix:laplace}, see in particular Remark \ref{rem:parameterlaplace}).
\end{proof}

\begin{rem}\label{rmk:dVcompact}
It is worth noting that $dV_{\nu,\hbar} = O(\hbar)$, so $\hbar^{-1} dV_{\nu,\hbar}$ lie in a compact set of smooth volume forms as $\hbar$ tends to 0 and $\nu$ ranges in a bounded interval.
\end{rem}

\subsection{The extremal envelope} 

The circle-invariant set-up allows us to identify the extremal
envelope and the forbidden region explicitly in terms of the moment
map, the potential and the Legendre transform.

In the neighbourhood $U$ of $Y$, note that the locally defined
expression $\ve t - u(\ve,\cdot)$ defines a potential on $L|U$. This is
because in $U$, $\ve t$ is a potential on $\cO(\ve Y)$ and $\eta_\ve =
- u(\ve,\cdot)$ is a potential on $L(-\ve Y)$, so their sum is a
potential on $L$.  Note further that 
$$
\mu(t,w)=\ve \Rightarrow \ve t - u(\ve,w) = \phi(t,w)
$$
by definition of the Legendre transform.  Define
\begin{equation}\label{e21.29.9.15}
\psi_\ve(t,w) = \left\{\begin{array}{l} \ve t -u(\ve,w)\mbox{ if
    }\mu(t,w) \leq \ve,\\
\phi(t,w)\mbox{ otherwise}\end{array}\right.
\end{equation}
This definition makes sense initially only in $U$, but can clearly be
extended to equal $\phi$ over $X\setminus U$.   So defined, $\psi_\ve$
is a continuous potential on $L$.

\begin{rem} The definition \eqref{e21.29.9.15} can be motivated as
  follows.  The simplest possible function with correct Lelong number
  has the form $\ve t + \mbox{const}$, where $t = \log|z|$ as
  before. This does not extend globally, so away from $Y$ we try to
  patch it to $\phi(t)$.  For fixed $w$, the slope of $\p_t\phi(t,w)$
  is equal to $\ve$ precisely when $\mu(t,w)=\ve$, since $\p_t\phi =
  \mu$.  Thus we extend $\ve t + \mbox{const}$ across $\mu^{-1}(\ve)$
  to equal $\phi$.  For this to be continuous, we need the constant to
  be equal to $-u(\ve,w)$, and we have arrived at \eqref{e21.29.9.15}.
\end{rem}

Then we have
\begin{thm} The function $\psi_\ve$ is the extremal envelope for
$(X,Y,L,\ve)$ (cf.\ Definition~\ref{def:extremalenvelope}).  Moreover,
for $(t,w)\in U$, 
$$\psi_\ve(t,w) < \phi(t,w) \mbox{ if and only if }
\mu(t,w) < \ve.
$$
In particular the forbidden region $D_\ve$ is equal to
$\mu^{-1}[0,\ve)$.
\end{thm}

\begin{proof}

We shall show first that $\psi_\ve(t,w) < \phi(t,w)$ iff $\mu(t,w)
<\ve$. This is a variant of the convexity argument used at the end of
the proof of Theorem~\ref{t1.14.2.13}.

Fix $w$ and let $t_\ve$ satisfy $\mu(t_\ve,w) =\ve$.  Then, from the definitions,
$$
\ve t - u(\ve,w) = \ve t - (\ve t_\ve - \phi(t_\ve,w)) = 
\ve (t - t_\ve) + \phi(t_\ve,w)).
$$
Then
$$
\ve t - u(\ve,w) - \phi(t,w)
= -(t-t_\ve)
\left(\ve - \frac{\phi(t_\ve,w) - \phi(t,w)}{t_\ve - t}\right).
$$
If $t<t_\ve$ then the difference quotient is strictly less than the
derivative at the upper end point $\phi'(t_\ve) = \ve$. Hence the
quantity in the large brackets is positive and so
$$
\ve t - u(\ve,w) < \phi(t,w)\mbox{ for }t< t_\ve
$$
as required.

We shall show next that $\psi_\ve$ is $C^1$.  The only issue is what
happens near $\mu^{-1}(\ve)$.   It is easier to use $(x,w)$ as local
coordinates.  Then
\begin{equation}
\psi_\ve = \ve u_x(x,w) - u(\ve,w)\mbox{ for }x<\ve
\end{equation}
and
\begin{equation}
\psi_\ve = xu_x(x,w) -u(x,w)\mbox{ for }x>\ve
\end{equation}
Since both expressions are equal for $x=\ve$ it follows that all
tangential derivatives agree on this hypersurface, whereas
$$
\lim_{x \to \ve^-}\p_x\psi_\ve(\ve,w) = \ve u_{xx}(\ve,w)
$$
and
$$
\lim_{x \to \ve^+}\p_x\psi_\ve(\ve,w) = \lim_{x\to \ve}
x u_{xx}(x,w) = \ve u_{xx}(\ve,w).
$$
Thus $\psi_\ve$ is $C^1$ as claimed.

A similar calculation, which we leave to the reader, shows that
$\psi_\ve$ is plurisubharmonic: from the regularity just proved, this follows by
showing that $\psi_\ve$ is plurisubharmonic on each side of the hypersurface
$\mu^{-1}(\ve)$.

Hence $\psi_\ve$ is a candidate for the extremal envelope $\phi_\ve$ in the sense
of Definition~\ref{def:extremalenvelope}.  We have to check that there is no better
candidate.  By definition, any other candidate must also equal $\phi$ on the
set $X\setminus \mu^{-1}[0,\ve)$.  

Suppose for contradiction $\psi_\epsilon<\phi_\epsilon$ at some point
$(t_0,w_0)\in U$.   Then there is a plurisubharmonic potential
$\gamma$ on $L$ bounded above by $\phi$ with Lelong number at least
$\epsilon$ along $Y$ such that
$\gamma(t_0,w_0)>\psi_\epsilon(t_0,w_0)$.  Since $\gamma\le \phi$ we
must then have $\mu(t_0,w_0)<\epsilon$ and so $\psi_\epsilon =
\epsilon t- u(\epsilon, w)$ near $(t_0,w_0)$.  Suppose that
$\gamma'(t_0,w_0)\ge \psi_\epsilon' = \epsilon$.  Along the image of
$t\mapsto (t,w_0)$ we have $\gamma'(t,w_0)$ is non-decreasing, so at
the point $(t,w)$ on $\mu^{-1}(\epsilon)$ we have
$\gamma>\psi_\epsilon(t,w) = \phi(t,w)$ which is absurd. Hence
$\gamma'(t_0,w_0)<\epsilon$.  Again by the monotonicity of $\gamma'$
along this image we see that $\gamma'<\epsilon$ for all $t<t_0$ and so
the Lelong number of $u$ is strictly less that $\epsilon$ which is
also absurd.  Hence such a $\gamma$ cannot exist, and we conclude
$\psi_\epsilon = \phi_\epsilon$ as desired. 
\end{proof}

\begin{cor} With notaton as above, the equilibrium set for $Y$ with
  respect to $\epsilon$ is the complement of $\mu^{-1}[0,\epsilon)$
  and the forbidden region is $\mu^{-1}[0,\epsilon)$ 
\end{cor}
\begin{proof}
Follows directly from the Theorem.
\end{proof}

\begin{rem}
Berman has proved that the extremal envelope $\phi_\ve$ is generally no better
than $C^{1,1}$, and our explicit formula \eqref{e21.29.9.15} displays precisely this
regularity.  On the other hand, we have seen in the course of the
proof that $\psi_\ve = \phi_\ve$ is conormal with respect to the
hypersurface $\mu^{-1}(\ve)$: that is to say
$$
V_1\ldots V_N \psi_\ve \in C^{1,1}
$$
for any number of vector fields $V_j$, provided that these are all
{\em tangential} to $\mu^{-1}(\ve)$.  It would be interesting to
investigate the conormal regularity of the extremal envelope in other situations.
\end{rem}

As an application of this explicit identification of the extremal
envelope, we prove the following technical result. 

\begin{lem}\label{lem:vanishingorder}
Suppose that $f \in H^{\epsilon'k}_{k\phi}(U)$  and
$$
\Omega'\Subset \mu^{-1}[0,\epsilon').
$$
Then there are constants $C$ and $c$ such that
\begin{equation}
\sup_{\Omega'}|f|_{k\phi} \leq Ce^{-ck}\|f\|_{k\phi,U}
\end{equation}
\end{lem}
\begin{proof}
The proof is similar to that of Proposition \ref{prop:bergmanboundextremal}.  
Let $N_k = \sup_{\Omega'} |f|_{k\phi}$ and set
$$
v := \frac{1}{k}\left(\log|f|_{k\phi} - \log N_k\right).
$$
Then $v \leq 0$ and $\phi + v$ is a competitor to be the envelope for
$\nu Y$ on $\Omega'$. Hence $\phi + v \leq \psi_{\epsilon'}$.  Rearranging this gives that over $\Omega'$
$$|f|_{k\phi} \leq N_k\exp(-k(\phi - \psi_{\epsilon'}))\le Ck^n \exp(-ck) \|f\|_{k\phi,u}$$
where we have used the $L^2$ implies $L^{\infty}$ bound from Proposition \ref{prop:Bochner} and set $c:\inf_{\Omega'}(\phi - \psi_\nu) = c$ which is strictly positive as $\Omega'\Subset \mu^{-1}[0,\epsilon')$.
\end{proof}

\subsection{A modified glueing result} \label{sec:modified}We will need a slight modification of our glueing result that relaxes the decay away from the diagonal condition in the presence of a holomorphic $S^1$-action.  

\begin{dfn}\label{def:S1local}
We say that $B_k^\epsilon$ is an \emph{$S^1$-local partial Bergman kernel} if it has the holomorphic and almost reproducing property as in Definition \ref{dfn:approxbergman} and the following decay away from the diagonal in standard coordinates:
$$|B_k^{\epsilon} (z,w,z',w')|_{k\phi} \le Ck^d e^{-c(\sqrt{k} |w-w'| + k|x-x'|^2)}\text{ for all } (z,w),(z',w')\in U$$
where, we recall, $x=\mu(z,w)$ and $x'=\mu(z',w')$ is the value of the
moment map at these points.
\end{dfn}

\begin{rem}
The above form of decay away from the diagonal may appear rather unusual, in that the decay is  faster in the directions normal to $Y$ than other directions.  We have stated it in this way simply because that is what our particular construction of $B_k^\epsilon$ satisfies.  The precise decay will not matter for our application, and it is sufficient to have something of the form $e^{O(\sqrt{k}) \rho((z,w),(z',w'))}$ where $\rho$ is a distance function (say the geodesic distance defined by a given K\"ahler metric).  In the above $|w-w'|$ refers the Euclidean norm with respect to our standard coordinates around $Y$ which we assume to exist locally (see Section \ref{sec:circlesetup}).  Again for our purpose one could, if one prefers,  replace this with $\rho_Y(w,w')$ where $\rho_Y$ is the geodesic distance with respect to some given K\"ahler metric on $Y$.
\end{rem}

\begin{thm} \label{thm:modifiedglue} With the setup as above, suppose $B_k^\epsilon$ is an $S^1$-local PBK of order $N$ on $W\subset U$, and suppose that $\epsilon$ is sufficiently small so that the forbidden region $D_\epsilon$ lies in $W$.  Then if $W'$ is an open, relatively compact subset of $W$ then we have for all $r\ge 0$ that
$$ K_k^\epsilon(x,y) = B_k^\epsilon(x,y) + O_{C^r}(k^{d/2+r/2-N}) \text{ for all } x,y\in W'.$$
\end{thm}
\begin{proof}
The proof is the same as that of Theorem \ref{thm:gluelocalBK}.  The only place in which we used the decay away form the diagonal was in \eqref{eq:useofdecaydiagonal}.  But if $(z,w)\in W'$ and $\overline{\partial} \chi(z')\neq 0$ at a point $(z',w')$ then $x$ and $x'$ are a bounded distance apart (as $\chi\equiv 1$ on $W$) and so \eqref{eq:useofdecaydiagonal} still holds.
\end{proof}

\subsection{The Local Partial Bergman Kernel} We are now ready to define our local PBK.  For $\nu\in [0,\epsilon')$ let $G_{n,k}$ denote the reproducing kernel on $Y\times Y$ for $L^k\otimes \mathcal O(-nY)|_Y$ with respect to the inner-product defined in \eqref{def:innerproduct}.
\begin{dfn}
Define
\begin{equation}
B_k^\epsilon: = \sum_{n=\epsilon k}^{\epsilon' k} G_{n,k} \sigma^n \boxtimes \overline{\sigma}^n.
\end{equation}
\end{dfn}
So by our conventions made following Lemma \ref{lem:expansion}, $B_k^\epsilon$ is a holomorphic section of $L^k|_U\boxtimes \overline{L}^k|_U$.

\begin{thm}\label{thm:constructionlocalPBK}
$B_k^\epsilon$ is an $S^1$-local PBK for $(\epsilon,Y)$ on $\Omega'$ (with respect to the chosen cutoff function $\chi$ used in \eqref{def:innerproduct}).
\end{thm}

Before the proof we make a convenient definition:

\begin{dfn}
In standard local coordinates let
$$U(\nu,x,w): = u(\nu,w) - u(x,w) - u_x(x,w)(\nu-x) \text{ for } \nu\in [0,\epsilon']$$
\end{dfn}

\begin{lem}\label{lem:propertyU}
\begin{enumerate}
\item There exists a constant $c>0$ such that $U(\nu,x,w)\ge c(x-\nu)^2$.
\item For any $f\in H^0(L(-\nu Y)^k|_Y)$ we have
$$ |f\sigma^n |_{k\phi} = |f|_{k\eta_\nu} e^{-kU(\nu,x,w)}.$$
\end{enumerate}
\end{lem}
\begin{proof}
The first statement follows as $U(\nu,x,w) = u_{xx}(w,x')(\nu-x)^2$ for some $x'$ which is strictly bounded from below as $x\mapsto u(x,w)$ is strictly convex.  The second statement is a simple calculation using the definitions and is left to the reader.
\end{proof}

\begin{proof}[Proof of Theorem \ref{thm:constructionlocalPBK}]
We first show that $B_k^\epsilon$ has the almost reproducing property, i.e. for some $c,C>0$ we have
\begin{equation}
 |f(z) - (B_{k,z}^\epsilon, \chi f)_{k\phi}|_{k\phi} \le Ce^{-ck} \|f\|_{k\phi,U}\label{eq:reproducingrepeat}
 \end{equation}
for all $f\in H_{k\phi}^{\epsilon k}(U)$ and $z\in \Omega'$.  For this we know from Lemma \ref{lem:expansion} that we can write such an $f$ as $f=g+s$ where
$$ s= \sum_{n>\epsilon' k} s_n\sigma^n \text{ and } g= \sum_{n=\epsilon k}^{\epsilon' k}g_n \sigma^n$$
and  $s_n,g_n \in H^0(L(-\nu Y)^k|_Y)$.  By integrating in the normal direction to $Y$ first, we see by construction that if $z\in \Omega'$ we have
$$ (\chi g, B_{k,z}^\epsilon)_{k\phi} = g(z) \text{ and }  (\chi s, B_{k,z}^\epsilon)_{k\phi} = 0.$$
On the other hand as $s$ vanishes to order at least $\epsilon' k$ along $Y$ we have from Lemma \ref{lem:vanishingorder} that $\sup_{\Omega'} |s|_{k\phi} \le C e^{-ck} \|s\|_{k\phi} \le Ce^{-ck} \|f\|_{k\phi}$ for some $c,C>0$.  Putting this together gives \eqref{eq:reproducingrepeat}.

Finally we prove that $B_k^\epsilon$ has the desired decay away from the diagonal property, i.e.\
$$ |B_k^\epsilon(z,w,z',w')|_{k\phi} \le Ck^d e^{-c(\sqrt{k}|w-w'| +k|x -
  x'|^2)} \text{ for all } (z,w),(z',w')\in U.$$
Fix $(z,w)$.  Then by Cauchy-Schwarz and the definition of the function $U$ (i.e.\ Lemma \ref{lem:propertyU})
\begin{align*}
|B_k^\epsilon(z,w,z',w')|^2_{k\phi}& \le O(k) \sum_{\epsilon k\le n\le \epsilon' k} |G_{n,k}(w,w')|^2_{k\eta_\nu} e^{-2k(U(\nu,x,w) + U(\nu,x',w'))} \\
&\le O(k)  \sum_{\epsilon k\le n\le \epsilon' k} |G_{n,k}(w,w')|^2_{k\eta_\nu} e^{-ck((x-\nu)^2 + (x'-\nu)^2)}
\end{align*}
for some $c>0$.  Recall that each $G_{n,k}$ decays exponentially fast
away away from the diagonal, 
\begin{equation}
|G_{n,k}(w,w')|_{k\eta_\nu} \le O(k^{2d-1})
e^{-c\sqrt{k}|w-w'|}\label{eq:decaydiagonalG}.
\end{equation} 
We have given a proof of this classical fact in
Theorem~\ref{thm:decaynonpartial}.  Moreover this estimate is 
uniform as $\nu$ ranges in a bounded interval, as follows directly
from our proof.  (The reader may have expected to see $O(k^{2d-2})$ in equation \ref{eq:decaydiagonalG} instead of $O(k^{2d-1})$ since since $2\dim Y =2d-2$, but the reason for this missing fact of $k$ is that whereas it is true that for a \emph{fixed} volume form on $Y$ the Bergman kernel has decay at rate $O(k^{2d-2})
e^{-c\sqrt{k}|w-w'|}$  (and this decay still holds for
volume forms that lie in a compact set), in the above $G_{n,k}$ is
taken with respect to the volume form $dV_{\nu,\hbar} = O(\hbar) =
O(k^{-1/2})$  which is shrinking with respect to $k$.)

Hence
$$|B_k^\epsilon(z,w,z',w')|_{k\phi}  = O(k^{2d}) O(e^{-c\sqrt{k}|w-w'|})  e^{-ck((x-\nu)^2 + (x'-\nu)^2)}$$
for some $c>0$.  On the other hand, by completing the square
$$ (x-\nu)^{2} + (x'-\nu)^{2} = \frac{1}{2}(x-x')^{2} + 2\left(\nu - \frac{x+x'}{2}\right)^{2}\ge \frac{1}{2}(x-x')^{2}$$
which gives the desired result.
\end{proof}

\section{Proof of Main Theorems}

We now put what we have done together to prove the main theorems.  The main task is to understand the asymptotics of our $S^1$-local PBK
\begin{equation}
B_k^\epsilon: = \sum_{n=\epsilon k}^{\epsilon' k} G_{n,k} \sigma^n \boxtimes \overline{\sigma}^n.
\end{equation}
The idea is to use the standard asymptotic expansion of the density functions on $Y$ to expand the functions $G_{n,k}$ in powers of $k$, and then use the Euler-Maclaurin formula to evaluate the sum in terms on an integral, which by Laplace's method can also be expanded in powers of $k$.  To display the main ideas and keep the proof short, we include an account of each of these standard techniques in the Appendix.\medskip

We start by recalling the well-known asymptotic expansion of the Bergman kernel.  Let $L'$ be an ample line bundle on a compact complex manifold $Y$ of dimension $d$ with positive hermitian metric $e^{-2\phi}$ and let $dV$ be a smooth volume form on $Y$.   These define an $L^2$-inner product on sections of $L'^k$ and we let $K'_k$ denote the reproducing kernel for the projection to the holomorphic sections, and let $\rho_k'(y) = K_k(y,y)$ be the corresponding density function.

\begin{thm}\label{thm:bergmanexpansion}
There exist smooth functions $a_0,a_1,\ldots$ on $Y$ such that for any $p,r\ge 0$ there is an asymptotic expansion
$$\rho'_k = a_0 k^d + a_1 k^{d-1} + \cdots + a_p k^{d-p} + O_{C^r}(k^{d-p-1}).$$
Furthermore the $a_i$ are universal quantities that depend smoothly on $dV$ and the curvature of $\phi$, in particular
\begin{equation}
 a_0 = \frac{(dd^c \phi)^{[d]}}{dV}.\label{eq:a0}
 \end{equation}
Moreover given any $c>0$ and background K\"ahler form $\omega_0$ on $Y$, the $O_{C^r}(k^{d-p-1})$ error term may be taken uniformly over all $(L', e^{-\phi})$ such that $dd^c\phi \ge c\omega_0$ as well as uniformly over all volume forms that lie within a compact set.
\end{thm}
\begin{proof} The first statement is the famous asymptotic expansion of the Bergman kernel, due to Fefferman \cite{F1}, Catlin \cite{C1}, Tian \cite{T1} and Zelditch \cite{Z1},   The statement that the error term may be taken uniformly follows from the same proofs (e.g.\ it is clear that this is the case for the local Bergman Kernel of \cite{BBS}, and the glueing theorem of \cite[Thm 3.1]{BBS} uses the H\"ormander estimate which gives a uniform error given the assumed bound on $dd^c\phi$, and for the variation with respect to $dV$ see \cite[Sec 2.4]{BBS} which is also clearly uniform as $dV$ varies in a compact set.
\end{proof}

\begin{proof}(Proof of Theorem \ref{thm:main1} and Theorem \ref{thm:main2})
First we note that $B_k^\epsilon$ is an $S^1$-local PBK which by Theorem \ref{thm:modifiedglue} approximates the globally defined PBK, so  in standard coordinates $(z,w)$ we have
$$\rho_k^\epsilon(z,w) = B_k^\epsilon (z,w,z,w) e^{-2k\phi(z,w)} + O(k^{-\infty}).$$
Hence the goal becomes to understand the asymptotics of the quantity
$$ B:=B_k^\epsilon (z,w,z,w) e^{-2k\phi(z,w)}.$$
From now on we will work with the variable $\hbar = k^{-1/2}$.  By Lemma \ref{lem:etanugloba} the $\eta_\nu$ have positive curvature bounded from below uniformly over $\nu\in [0,\epsilon']$.  Observe that $\eta_{\nu}$ are smooth in $\nu$ and recall that $\hbar^{-1 }dV_{\nu,\hbar}$ are volume forms that lie in a compact set (and moreover are smooth in $\hbar$ and $\nu$).  Thus Theorem \ref{thm:bergmanexpansion} implies there are smooth functions $a_i(\nu,w)$ on $[0,\epsilon']\times Y$  such that that for any $r,p\ge 0$ we have
$$\hbar G_{n,k}(w,w)e^{-2\hbar^{-2}\eta_\nu(w)} = a_0(\nu,w)\hbar^{2-2d} + a_1(\nu,w)\hbar^{4-2d} + \cdots  + a_p(\nu,w) \hbar^{2p+2-2d} + O_{C^r}(\hbar^{2p+4-2d})$$
where, as usual, $\nu  = n/k = \hbar^2 n$.  To capture this information we define
$$ \alpha(\nu,\hbar^2,w): = a_0(\nu,w)  + a_1(\nu,w) \hbar^2+ \cdots + a_p(\nu,w) \hbar^{2p}$$
which is clearly smooth in all variables.

\begin{rem}\label{rem:mu00}
For later use, observe that from the leading order term of $dV_{\nu,\hbar}$ given in  \eqref{eq:dV} and the leading order term of the asymptotic expansion of the Bergman function given in \eqref{eq:a0} we have that
$$\alpha(0,0,w) = a_0(0,0) = \sqrt{\frac{u_{xx}(0,w)}{2\pi}}.$$
\end{rem}

Now using Lemma \ref{lem:propertyU} (i.e.\ the properties of our function $U$) gives
\begin{align*}
B &= \sum_{n=\epsilon k}^{\epsilon' k} G_{n,k}(w,w) |z|^{2n} e^{-2k\phi(z,w)}  \\
&=\sum_{n=\epsilon k}^{\epsilon' k} G_{n,k}(w,w) e^{-2k\eta_\nu(w)} e^{-2k U(\nu,x,w)} \\
&=\hbar^{1-2n}\left( \sum_{n=\epsilon k}^{\epsilon' k} e^{-2k U(x,\nu,w)} \alpha(\hbar^2 n,\hbar^2,w) + O_{C^r}(\hbar^{2p+2})\right).
\end{align*}
Now this sum can be calculated using the Euler-Maclaurin formula (see \ref{prop:emappendix}).  To state this let 
$$ q_{\hbar}(s) : = q(s,\hbar,x,w) = e^{-\hbar^{-2} U(x-\hbar s,x,w)}\alpha(x-\hbar s,\hbar^2,w)$$
and set
$$\xi = \frac{x-\epsilon}{\h}.$$
Then
  \begin{equation}\label{eq:em:use}
    \hbar^{2n} B = \int_{-\infty}^\xi q_{\hbar}(s) ds + \sum_{j=0}^{m-1} A_j \hbar^{j} + O(\hbar^m)
      \end{equation}
where 
\begin{align*}
A_0  = \frac{1}{2} q_{\hbar}(\xi) \text{ and } A_j = (-1)^{j-1}\frac{\beta_{j+1}}{(j+1)!} q_{\hbar}^{(j)}(\xi)
\end{align*}
where $\beta_j$ are the Bernoulli numbers.  Moreover one sees directly that the coefficients $A_j$ lift to the real blowup (see Appendix \ref{appendix:blow-up}) as in the statement of the Theorem.

Thus it remains only to analyse the integral in the right hand side of \eqref{eq:em:use} which can be done by Laplace's method.    In slightly more detail observe that $q_{\hbar}(s)$ has its unique critical point at $s=0$.  We rewrite the integral as
$$I:=\int_{-\infty}^{\zeta} q_{\hbar}(s) ds = \hbar^{-1} \int_{-\infty}^{x-\epsilon} q_{\hbar}\left(\frac{s}{\hbar}\right) ds$$
So this is precisely the setup of Laplace's method as discussed in Appendix \ref{appendix:laplace} which shows that this integral has an asymptotic expansion in powers of $\hbar$, in fact
$$ I = \frac{\sqrt{2\pi} \alpha(0,0,w)}{\sqrt{u_{xx}(0,w)}} \Phi(\sqrt{u_{xx}(0,w)} \zeta) + O(\hbar)$$
where as usual
$$\Phi(x) := \frac{1}{\sqrt{2\pi}} \int_{-\infty}^x e^{-\frac{t^2}{2}} dt.$$
Recall that  $u_{xx}(0,w) = |v|$ where $v$ is the generator of the $S^1$-action.  Thus plugging in (Remark \ref{rem:mu00})
$$\alpha(0,0,w) = \sqrt{\frac{|v|}{2\pi}},$$ as simple change of variables gives
$$I =  \frac{1}{\sqrt{2\pi |v|}} \int_{-\infty}^{\frac{x-\epsilon}{\hbar}} e^{-\frac{t^2}{2|v|^2}} dt + O(\hbar)$$
as required in the statement of Theorem \ref{thm:main2}.  Finally Theorem \ref{thm:main1} follows immediately from this by the behaviour of $\Phi(x)$ on the sets $\{x<\epsilon\}$ and $\{x>\epsilon\}$.  
\end{proof}
%Recall that
%$$ \hbar := \frac{1}{\sqrt}{k}.$$

\section{An Application}

We end with an application of our main theorem to the study of a certain natural function introduced by Ross-Witt Nystr\"om \cite{Ross} that one can associate to a divisor on a K\"ahler manifold.\medskip

Fix a line bundle $L$ on a compact complex manifold $X$ with hermitian metric $e^{-\phi}$.   Then the order of vanishing of sections along a divisor $Y$ determines a finite length filtration
$$ H^0(L^k) \supset H^0(L^k\otimes \mathcal I_Y) \supset H^0(L^k\otimes \mathcal I_Y^2)\supset \cdots \supset \{0\}.$$
Let $\ord_Y(s)$ denote the order of vanishing of a section $s$ along $Y$, and suppose that $\{s_{\alpha,k}\}$ is an $L^2$-orthonormal basis for $H^0(L^k)$ that is compatible with this filtration, i.e.\ for  each $j$ the set
$$\{ s_{\alpha,k} : \ord_Y(s_{\alpha,k})\ge j\}$$
is a basis for $H^0(L^k\otimes \mathcal I_Y^j)$, 

\begin{dfn}
Define $M_k\colon X\to \mathbb R$ by
$$M_k(z) = \frac{\sum_{\alpha} \ord_Y(s_{\alpha,k}) |s_{\alpha,k}|^2_{k\phi} }{k\sum_{\alpha} |s_{\alpha,k}|^2_{k\phi}}.$$  
\end{dfn}
One can check directly that this definition does not depend on choice of compatible orthonormal basis, and thus defines a natural smooth function on $X$ associated to $Y$ and the hermitian metric on $L$.

\begin{thm}\label{thm:application}
Suppose $\phi$ and $Y$ are invariant under an $S^1$ action on $(X,L)$.   Then there is a neighbourhood of $Y$ and an asymptotic expansion
\begin{equation}
M_k = c_0 + c_{1/2} k^{-1/2}  + c_1 k^{-1} +  \cdots + c_r k^{-N} + O_{C^r}(k^{-N-1/2})\label{eq:expandMk}
\end{equation}
where $c_i$ are smooth functions defined on this neighbourhood.   Moreover $c_0$ is precisely the hamiltonian of the $S^1$ action normalized so $Y = c_0^{-1}(0)$.
\end{thm}

\begin{rem}
It is shown in \cite[Sec 8]{Ross}, without any assumption of the existence of an $S^1$-action, that the limit 
$$\mu := \limsup_{k\to \infty} M_k$$
 converges almost everywhere on $X$.  We do not know anything about the regularity of $\mu$ in general, but it does have a ``push-forward'' property analogous to that of the Duistermaat-Heckman formula \cite[Thm.\ 8.1,8.3]{Ross} with the moment polytope being replaced by the Okounkov body of $(X,L)$.    The previous theorem shows that in the $S^1$-invariant case we in fact have that $\mu=c_0$ is the Hamiltonian in a neighbourhood of $Y$ which reaffirms this property.
\end{rem}

%\begin{rmk}
%  Of course $M_k$ is norm of the restriction to the diagonal of the section of %$L\boxtimes \overline{L}$ on $X\times X$ given by
%  \begin{equation}
%\tilde{M}_k=\sum_{\alpha} \ord_Y(s_{\alpha}) s_{\alpha,k} \boxtimes \overline{s%}_{\alpha,k}.\label{eq:expandMk}
%\end{equation}
%Observe however that the linear map $s\mapsto (s,\tilde{M}_k)_{k\phi}$ from $\G%amma(L^k)$ to $H^0(L^k)$ is not a projection operator, and there seems no reaso%n to expect that $M_k$ has an asymptotic expansion globally.
%\end{rmk}

\begin{proof}[Proof of Theorem~\ref{thm:application}]
For simplicity write $s_{\alpha}$ for $s_{\alpha,k}$ and set $n_{\alpha} = \ord_Y(s_{\alpha})$.  We also write $x\colon X\to \mathbb R$ for the hamiltonian of the $S^1$-action normalized so $Y=\{x=0\}$.    So the partial Bergman kernel for $(\epsilon, Y)$ is given by
$$ \rho_k^{\epsilon} = \sum_{n_{\alpha} \ge \epsilon k} |s_{\alpha}|_{k\phi}^2 \quad\text{ for }\epsilon k\in\mathbb N.$$
By Theorem~\ref{thm:main2} there is a neighbourhood of the divisor and an expansion of the partial density in powers of $k^{1/2}$:
$$ \rho^{\epsilon}_k = b_0(\epsilon)k^d  +  b_{1/2}(\epsilon) k^{d-1/2} + \cdots + b_N(\epsilon) k^{d-N} + O(k^{d-N-1})$$
where the $b_i(\epsilon)$ are smooth functions on $X$ with
$$ b_0(\epsilon) = \frac{1}{\sqrt{2\pi x}} \int_{-\infty}^{\sqrt{k}(x-\epsilon)} e^{-\frac{t^2}{2x}} dt.$$
Now
$$ \sum_{j=1}^{k\sigma} \rho_k^{\frac{j}{k} Y} = \sum_{j=1}^{k\sigma} \sum_{n_{\alpha}\ge j} |s_{\alpha}|^2_{k\phi} = \sum_{n_{\alpha}\le k\sigma} n_{\alpha} |s_{\alpha}|_{k\phi}^2 + \sum_{n_{\alpha}>k\sigma} k\sigma  |s_{\alpha}|_{k\phi}^2.$$ 
Thus on a smaller neighbourhood of $Y$ (say where $x\le \sigma/2$) we have
$$M_k= \frac{1}{k\rho_k}\sum_{j=1}^{k\sigma} \rho_k^{\frac{j}{k} Y}  + O(k^{-\infty})$$
where $\rho_k$ is the usual Bergman function on $X$.  Now $\rho_k = k^d + O(k^{d-1})$ has a global asymptotic expansion in powers of $k$, and thus the Euler-Maclaurin formula gives the required expansion \eqref{eq:expandMk} for $M_k$ with
$$ c_0 = \int_0^{\sigma} b_0(s) ds.$$
Roughly speaking, for $k$ large $b_0(s)$ is approximately $1$ for $x>s$ and $0$ for $x<s$, and so $c_0=\int_0^s b_0(s) ds \sim \int_0^x ds = x$.  We claim that in fact for $x<\sigma/2$ we have
$$ \int_0^\sigma  b_0(s) ds = x + O(k^{-\infty})$$
To see this, integrate by parts to get
$$ I:= \int_0^\sigma b_0(s) ds = [ sb_0(s) ]_{s=0}^{\sigma}+ \sqrt{k} \int_0^\sigma \frac{s}{\sqrt{2\pi x}} e^{\frac{-k(x-s)^2}{2x}}ds$$
Since $x<\sigma/2$ the boundary term is  $O(k^{-\infty})$, and by direct calculation
\begin{align*}
  k^{-1/2} I &=  x \int_0^\sigma \frac{1}{\sqrt{2\pi x}} e^{-\frac{k(x-s)^2}{2x}}ds - \int_0^{\sigma} (x-s)  \frac{1}{\sqrt{2\pi x}}e^{-\frac{k(x-s)^2}{2x}}ds + O(k^{-\infty})\\
  &= k^{-1/2} x + O(k^{-\infty})
\end{align*}
as claimed.
%% Note to self: I think a more refined statement is that this is true up to $O^e^{-kx}$ -- JR
\end{proof}

\appendix

\section{The Real Blow-up}\label{appendix:blow-up}

For the reader's convenience we collect here some elementary properties of a special case of the real blow-up.   More details and constructions of greater generality can be found in Melrose's book \cite[Chapter 5]{Melrose}, as well as in \cite{Hassel} and \cite[Section 2.3]{Greiser}.    

\begin{dfn}
Let $S_{+}^1 = \{ e^{i\theta} : \theta\in [0,\pi]\}\subset \mathbb R^2$ be the upper-semicircle.  The \emph{real blow-up} of the upper-half space $\mathbb R\times \mathbb R_{\ge 0}$ at the point $p=(0,0)$ is defined to be 
$$[\mathbb R\times \mathbb R_{\ge 0}; p]:= \mathbb R_{\ge 0} \times S_{+}^1$$
along with the \emph{blow-down} map
$$\beta\colon [\mathbb R\times \mathbb R_{\ge 0}; p]\to \mathbb R\times \mathbb R_{\ge 0}\quad  \beta(r,e^{i\theta}) = re^{i\theta}.$$
\end{dfn}
Thus $[\mathbb R\times \mathbb R_{\ge 0}; p]$ is the upper half-space with a copy of the semicircle $S^1_+$ inserted at the origin, which we consider as a manifold-with-corners.  Then $\beta$ is smooth, and moreover is a diffeomorphism away from the \emph{exceptional set} $E= \beta^{-1}(p) =\{ (0,e^{i\theta}) : \theta\in [0,\pi]\}$.     The boundary of $ [\mathbb R\times \mathbb R_{\ge 0}; p]$ has three pieces, namely $E$ and the two axes $H_{\pm} = \{ (r,\pm 1) : r\in \mathbb R_{\ge 0}\}$, and two corners $C_{\pm} = H_{\pm}\cap E$. 

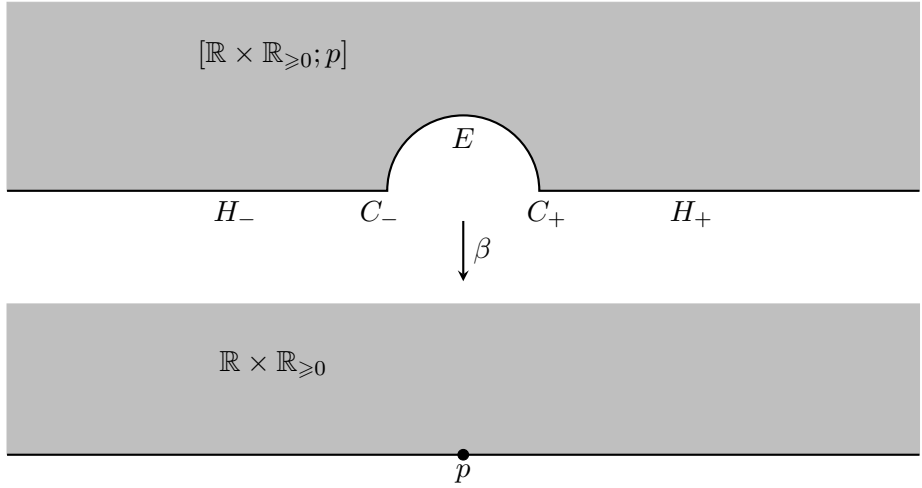
\begin{center}
\begin{tikzpicture}[>=stealth,domain=-4:4,smooth,scale=1]%,rotate=90, xscale = -1]
%
%  Shading in the blow-up
%
\filldraw[color = lightgray] 
(1,0) arc (0:180:1) -- (-6,0) --
(-6,2.5) -- (6,2.5) --  (6,0mm) -- (1,0);
\draw (-2.5,1.8) node  {$[\mathbb R\times \mathbb R_{\ge 0};p]$};
\draw (0,.7) node {$E$};
%\draw[very thin,color=gray] (-0.1,-1.1) grid (3.9,3.9);
%\draw[thick] (-4,0) -- (-1.,0) node[right] {$x$}; 
%\draw[thick]  (1,0) -- (4,0) node[above] {somthing};
%\draw[color=red]	plot (\x,{1.2/(1+\x*\x/2.88)});
\draw[thick, color=black] (6,0) --(1,0) arc (0:180:1) -- (-6,0);
%\draw[thick,color=magenta] (0,1) -- (0,2.5); 
%\draw[thick, color = black, domain = -0.9999:0.9999,smooth] plot(\x, {(1
%  - \x*\x)^{0.5}})  ;
\filldraw[color = lightgray]  [yshift = -3.5cm]
(-6,0) --
(-6,2) -- (6,2) --  (6,0mm) -- (-6,0);
%\draw [yshift=-3.5cm,thin,red] (-4,0.2) -- (4,0.2);
\draw [yshift=-3.5cm](-2.5,1.2) node  {$\mathbb R\times \mathbb R_{\ge 0}$};
\draw[thick, color = black] [yshift=-3.5cm] (-6,0) -- (6,0);
%\draw (2.4,0) [yshift = -3.5cm]node[below]{$X$};
\filldraw (0,0) [color=black,yshift = -3.5cm] circle (2pt) node[below]{$p$};
%  $I$ interval
%
% \draw[thick,color = black] [xshift = 5.1cm] (0,0) -- (0,2.5);
% \draw[color = black] [xshift = 5.1cm] (0, -0.05) node[right]{$\epsilon = 0$};
% \filldraw[color = black] [xshift = 5.1cm] (0,0) circle (2pt);
% \filldraw[color = red] [xshift = 5.1cm] (0,0.2) circle (2pt);
% \draw[color = red] [xshift = 5.1cm] (0,0.35) node[right]{$\epsilon =
%   0.2$};
% \filldraw[color = black] [xshift = 5.1cm,yshift = -3.5cm] (0,0) circle (2pt);
% \filldraw[color = red] [xshift = 5.1cm,yshift=-3.5cm] (0,0.2) circle (2pt);
% \draw [xshift = 5.1cm,yshift=-3.5cm] (0,0) node[right]{$\epsilon =
%   0$};
% \draw[color = black] [xshift = 5.1cm] (0,2.45) node[right]{$I$};
% \draw[color = black] [xshift = 5.1cm,yshift = -3.5cm] (0,2.45)
% node[right]{$I$};
% \draw[color = black,thick] [xshift = 5.1cm,yshift = -3.5cm] (0,0) --
% (0,2);
%\draw[color = red,yshift=-3.5cm, xshift = 5.1cm] (0,0.35) node[right]{$\epsilon =
%  0.2$};
% %
%  Mappings
%
% Blow-down
\draw[->] [thick] (0,-.4) -- (0,-1.2);
\draw (0, -0.8) node[right] {$\beta$} ;
%
% phi
%
% \draw[->] [thick] (4.2,1.25) -- (5.0, 1.25);
% \draw (4.6, 1.25) node[above] {$\pi_1$} ;
% \begin{scope} [yshift = -3.5cm]
% \draw[->] [thick] (4.2,1.25) -- (5.0, 1.25);
% \draw (4.6, 1.25) node[above] {$\pi_0$} ;
% \end{scope}
%
%  Labelling faces of blow-up
%

% Labelling the corners
%
\draw (0:1.1)  node [below]{$C_+$} ;

\filldraw (180:1.1) node [below]{$C_-$};
\filldraw (180:-3.0) node [below]{$H_+$};
\filldraw (180:3.0) node [below]{$H_-$};

\end{tikzpicture}
\end{center}

 For higher dimensions, we let the real blow-up of $\mathbb R^n\times \mathbb R_{\ge 0}$ along $Y:= \mathbb R^{n-1} \times \{p\}$ be $[\mathbb R^n \times \mathbb R_{\ge 0}; Y ] := \mathbb R^{n-1} \times [\mathbb R\times \mathbb R_{\ge 0};p]$, and by patching in the obvious way, this extends to define the real blow-up $[X\times \mathbb R_{\ge 0}; D\times \{0\}]$ for any smooth real manifold $X$ and codimension 1 smooth divisor $D$.

\medskip

% If $A\subset \mathbb R\times \mathbb R_{\ge 0}$ is a submanifold other than $\{p\}$ then the \emph{real blow-up} $[A;p]$ of $A$ at $p$ (also called the \emph{proper transform} of $A$) is defined to be the closure of $\beta^{-1}(A-\{p\})$.  

Let $x$ be the standard coordinate on $\mathbb R$ and $\hbar$ the standard coordinate on $\mathbb R_{\ge 0}$.  Then the functions 
$$\xi: = \frac{x}{\hbar} \quad \text{ and }\quad \hbar$$ 
give coordinates on the blow-up $B:=[\mathbb R\times \mathbb R_{\ge 0};p]$ on an open set containing the interior of $E$ (precisely, $\xi= \operatorname{cotan} \theta$ and $\hbar = r\sin \theta$ for $(r,\theta) \in B$ with $\theta\neq 0$).  Moreover in these coordinates $\beta(\xi,\hbar) = (\xi \hbar,\hbar)$.   Around each of the corners $C_{\pm}$ we have coordinates given by 
$$x \quad \text{ and } \quad \eta_{\pm}=\pm\frac{\hbar}{x}$$ 
(so $x=r\cos \theta$ and  $\eta_{\pm} = \pm \tan \theta$) and in these coordinate $\beta(x,\eta_{\pm}) = (x,\pm x\eta_{\pm})$.  \medskip

Our interest will be in functions on $\mathbb R\times \mathbb R_{>0}$ that lift to smooth functions on the blow-up (i.e.\ that extend to a smooth function right up to the boundary).  As a model example (which will be considered in more generality below) let
$$ q(x,\h) = e^{-\h^{-2} x^2}$$
which is clearly smooth for $\h>0$, and extends smoothly to the boundary away from $p:=(0,0)$.  We claim the lift $\beta^*g := q\circ \beta$ extends to a smooth function on all of $[\mathbb R\times \mathbb R_{\ge 0}; p]$.   To see this observe that in the coordinates $(\xi,\h)$ we have
$$ \beta^* q(\xi,\h) = q(\xi \h,\h) = e^{-\xi^2}$$
which is clearly smooth, and in coordinates $(x,\eta_{\pm})$,
$$ \beta^* q(x,\eta_\pm) = q(x,\pm x \eta_\pm) = e^{\eta_\pm^{-2}}$$
which extends to a smooth function for all $\eta_\pm\ge 0$ by declaring it to be zero when $\eta_\pm =0$.  

\section{The Euler-Maclaurin Formula}\label{sec:em}

%\begin{thm}[Euler-Maclaurin]
%Suppose $g\colon \mathbb R\to \mathbb R$ is smooth.  Then for all integers $a\ge b$ and all $m\ge 1$ we have
%$$\sum_{j=a}^b g(j) = \int_{a}^b g(t) dt + \sum_{i=1}^m (-1)^k \frac{B_i}{i!} \left[g^{(i-1)}(t)\right]_a^b + (-1)^{m+1} \int_{a}^b \frac{B_m(\{1-t\})}{m!} f^{(m)}(t) dt,$$
%where $B_k$ are the Bernoulli numbers given by $B_k = B_k(0)$ where
%$$ \frac{e^{zx}}{e^z-1} = \sum_{n\ge 0} B_n(x) \frac{z^n}{n!}$$
%and $\{t\}$ denotes the fractional part of $t$.%
%\end{thm}
%The goal of this section is to use the Euler-Maclaurin formula to compute the sum appearing in the local partial Bergman kernel considered in \eqref{e18.23.2.13}.  It is convenient from this point forward to use the change of variables
%$$ \h = \frac{1}{\sqrt{k}}$$
%as discussed in the introduction.\medskip

Let $U(\nu,x,w)$ be a smooth function, where $\nu\in [0,\epsilon']$, $x\in \mathbb R_{\ge 0}$ and $w$ ranges in an open relatively compact subset of $\mathbb C^m$.  Suppose that $U$ is such that $x\mapsto U(\nu,x,w)$ is strictly convex and 
\begin{equation}
 U(x,x,w) =0 \quad\text{and }\quad U_x(x,x,w)=0 \quad \text{ for all } x,w.\label{eq:assumptionsconvex} 
\end{equation}
Also let $\alpha(\nu,\h^2,w)$ be a bounded smooth function where $\h^2 \in \mathbb R_{\ge 0}$ and fix also $\epsilon\in [0,\epsilon')$.  For $\h>0$ let
$$ p_{\h}(t) := p(t,\h,x,w) := e^{-\h^{-2} U(\h^2 t,x,w)} \alpha(\h^2 t,\h^2,w)$$
Then our interest is in the behaviour of the sum
\begin{equation}
S:=S(\hbar,x,w):= \h\sum_{n=\epsilon k}^{\epsilon'k} e^{-\h^{-2}U(\h^2 n,x,w)} \alpha(\h^2 n,\h^2,w) = \hbar \sum_{n=\hbar^{-2}\epsilon}^{\hbar^{-2}\sigma} p_{\hbar}(n) \label{eq:sumgoal}
\end{equation}
as $\hbar$ tends to zero.  
%so the sum we are interested in is
%$$ S =.$$

We recall the Bernoulli numbers are defined by $\beta_j = \beta_j(0)$ where
$$ \frac{ze^{zx}}{e^z-1} = \sum_{j\ge 0} \beta_j(x) \frac{z^j}{j!}.$$
\begin{thm}[Euler-Maclaurin Formula]
Suppose $p\colon \mathbb R\to \mathbb R$ is smooth and $p^{(i)}(t)$ tends to zero as $t$ tends to infinity for all $i\ge 0$.    Then for any integers $a$ and $m\ge 1$ we have
$$\sum_{n=a}^\infty p(n) = \int_{a}^\infty p(t) dt + \frac{1}{2} p(a) - \sum_{j=1}^{m-1} \frac{\beta_{j+1}}{(j+1)!} p^{(j)}(a) -  \int_{a}^\infty \frac{\beta_{m}(\{1-t\})}{m!} p^{(m)}(t) dt,$$
where $\{t\}$ denotes the fractional part of $t$.
\end{thm}
\begin{proof}
  This is proved, for example, in \cite[Theorem 5]{Stalker}.
\end{proof}

Since $x\mapsto U(\nu,x,w)$ has its (unique) critical point at $x=\nu$ it is convenient to make a further change of variables by setting
$$ q_{\hbar}(s) = p_{\hbar}\left(\frac{x-\hbar s}{\hbar^2}\right).$$
Observe then that from \eqref{eq:assumptionsconvex} 
$$q_{\h}(s)= e^{-\hbar^{-2} U(x-\hbar s,x,w)}\alpha(x-\hbar s,x,w) = e^{-U_{xx} (x,x,w)s^2 + O(\hbar)} \alpha(x-\hbar s,x,w)$$
which has bounded derivatives with respect to $s$ uniformly as $\h$ tends to $0$.

\begin{prop}\label{prop:emappendix}
Set
$$\xi = \frac{x-\epsilon}{\h}.$$
Then for any  $m\ge 1$ and all $x<\epsilon'$ we have
  \begin{equation}\label{eq:em}
     \h\sum_{n=\epsilon k}^{\epsilon'k} e^{-\h^{-2}U(\h^2 n,x,w)} \alpha(\h^2 n,\h^2,w)= \int_{-\infty}^\xi q_{\hbar}(s) ds + \sum_{j=0}^{m-1} A_j \hbar^{j} + O(\hbar^m)
      \end{equation}
where 
\begin{align}
A_0 & = \frac{1}{2} q_{\hbar}(\xi) \\
 A_j &= (-1)^{j-1}\frac{\beta_{j+1}}{(j+1)!} q_{\hbar}^{(j)}(\xi)
\end{align}
which are bounded independent of $\h$.  Moreover this is uniform as $x$ varies in a compact subset of $[0,\epsilon')$.
\end{prop}
\begin{proof}
We can assume that $U$ is the restriction of a function defined for all $\nu\ge 0$ (which for simplicity we also denote by $U$) with the same convexity properties.   Similarly we assume $\alpha$ extends to a bounded function defined for all $\nu\ge 0$.   Now using the convexity of $U$ one sees that since $x\le \epsilon<\epsilon'$ we have
$$ p_{\hbar}(n) \le O(e^{-c\hbar^{-2}} e^{-c\hbar^{-2}(\epsilon'-\hbar^{-2}n)})$$
for some $c>0$.  Hence $\sum_{n>\hbar^{-2}\epsilon} p_{\hbar}(n) = O(h^{\infty})$ giving
$$ S = \h\sum_{n\ge \hbar^{-2}\epsilon} e^{-\h^{-2}U(\h^2 n,x,w)} \alpha(\h^2 n,\h^2,w) + O(\hbar^{\infty}).$$
Similarly one verifies that for $x\le \epsilon$ the derivatives of $p_{\hbar}$ tend to zero as $t$ tends to infinity.
Hence the above Euler-Maclaurin formula \eqref{eq:em} (applied with $a= \h^{-2}\epsilon$) gives
$$\hbar^{-1} S = \int_{\h^{-2}\epsilon}^{\infty} p_{\hbar}(t) dt + \frac{p_{\hbar}(\h^{-2}\epsilon)}{2} - \sum_{j=2}^{m-1} \frac{\beta_{j+1}}{(j+1)!} p_{\hbar}^{(j)}(\h^{-2} \epsilon) - \int_{\h^{-2}\epsilon}^{\infty} \frac{\beta_{m}(\{1-t\})}{m!} p_{\hbar}^{(m)}(t) dt + O(\hbar^{\infty}).$$
Now by instant computation
$$ q_{\hbar}^{(j)}(\xi) = (-\hbar)^{-j} p_\hbar^{(j)}(\hbar^{-2}\epsilon).$$
On the other hand by a change of variables
$$ \int_{\hbar^{-2}\epsilon}^\infty p_{\hbar}(t) dt = \hbar^{-1} \int_{-\infty}^{\xi}q_{\hbar}(s) ds$$
and a similar change of variables shows that
$$\int_{\h^{-2}\epsilon}^{\infty} \frac{\beta_{m}(\{1-t\})}{m!} p_{\hbar}^{(m)}(t) dt = O(\hbar^{m-1}).$$
 Putting this together gives the statement, and the fact that the $A_j$ are bounded as $\hbar$ tends to zero follows as the derivatives of $q_{\hbar}$ are bounded.
%Applying the change of variables from $t$ to $s$ and observing that $s=-\infty$ when $t=\infty$ and $s=\xi$ when $t=\h^{-2}\epsilon$ gives \eqref{eq:em}.  The statement about the boundedness of $A_j$ and $R_m$ is immediate 
\end{proof}
%\begin{rmk}
%  In the above proposition we observe the terms blah and blah are bounded independent of $\h$.
%\end{rmk}

\section{Laplace's Method}\label{appendix:laplace}

We now give an account of Laplace's method \cite{Laplace} tailored to our requirements (see also \cite{Zworski}).  We are concerned with obtaining the large-$k$ asymptotic behaviour of integrals of the form
\begin{equation}\label{e1.14.8.13}
F_k(x) = \sqrt{\frac{k}{2\pi}} \int_{-\infty}^x e^{-kf(t)}\alpha(t)\,\rd t
\end{equation}
where $\alpha$ is smooth and has compact support, and $f$ has the properties
\begin{itemize}
\item $f(t) \geq 0$;
\item $f(0) = 0$, $f'(0) =0$, $f''(0) = c>0$;
\item $f(x) \geq f_1 > 0 \mbox{ for all } |x|> \delta$,
\end{itemize}
where $\delta>0$ is some small number.

These properties ensure that if the support of $\alpha$ does not contain
$0$, then $F_k(x)$ is uniformly exponentially small in $k$,
independent of $x$. So the interesting case is that $\supp(\alpha) \ni
0$ and one expects the asymptotic expansion in negative powers of $k$
to see only the jets of $\alpha$ and $f$ at $0$.  The model problem is the
integral
\begin{equation}\label{e2.14.8.13}
Z_k(x) = \sqrt{\frac{k}{2\pi}} \int_{-\infty}^x e^{-kt^2/2}\,\rd t.
\end{equation}
This is the case $f(t) = t^2/2$, which clearly verifies the above
properties (even if the function `$1$' replacing $u$ in \eqref{e1.14.8.13} does not have
compact support).

Make the change of variables
\begin{equation} \label{e3.14.8.13}
y=\frac{x}{\hbar} \text{ and } s =\frac{t}{\hbar} \text{ and } \hbar = \frac{1}{\sqrt{k}}
\end{equation}
in \eqref{e2.14.8.13}
so
\begin{equation}\label{e4.14.8.13}
Z_k(\hbar y) = \frac{1}{\sqrt{2\pi}}\int_{-\infty}^y
e^{-s^2/2}\,\rd s = \Phi(y)
\end{equation}
where $\Phi$ is the usual normal distribution function.  In other
words,
\begin{equation}\label{e4a.14.8.13}
Z_k(x) = \Phi(x/\hbar).
\end{equation}

Note that if $x>0$ is fixed, $x/\hbar \to + \infty$ as $\hbar \to 0$,
so $Z_k(x) \to 1$ and in fact $Z_k(x) -1$ is exponentially small in
$k$ for fixed $x$. Similarly, if $x<0$ is fixed, $Z_k(x)$ is
exponentially small in $k$.  But \eqref{e4.14.8.13} shows that this
apparently discontinuous behaviour in $x$ near $0$ can be `smoothed'
by introducing a new variable $x/\hbar$, which corresponds to a real
radial blow-up, as discussed in Appendix~\ref{appendix:blow-up}.

Life is more interesting for general exponents $f$
satisfying the above conditions (and for functions $\alpha$ not identically
equal to $1$).  In that case, if $x>0$, for example, $F_k(x)$ is
exponentially close to
\begin{equation}\label{e5.14.8.13}
F_k(\infty) := 
\sqrt{\frac{k}{2\pi}} \int_{-\infty}^\infty e^{-kf(t)}\alpha(t)\,\rd t
\end{equation}
and the asymptotic expansion of this is given by Laplace's method.

\subsection{Incomplete Gaussian integrals}

We change notation slightly and set
\begin{equation}\label{e1a.18.8.13}
Z(x,\hbar;\alpha) = \frac{1}{\sqrt{2\pi}\hbar}\int_{-\infty}^x
e^{-t^2/2\hbar^2}\alpha(t)\,\rd t.
\end{equation}
We assume that $\alpha$ is $C^\infty$ and that 
derivatives of $\alpha$ are bounded,
\begin{equation}\label{e1.16.8.13}
\|\alpha\|_{C^r} \leq A_r
\end{equation}
for each $r\geq 0$.

For such a function, define
\begin{equation}\label{e2.16.8.13}
\delta_0\alpha(t) = \frac{\alpha(t)-\alpha(0)}{t} \mbox{ for }t\neq 0,\; \delta_0\alpha(0)
= \alpha'(0).
\end{equation}
Then $\delta_0\alpha$ is smooth and all derivatives are bounded.
We have the formula
\begin{equation}\label{e2a.16.8.13}
\delta_0\alpha(t) = \int_0^1 \alpha'(\lambda t)\,\rd \lambda.
\end{equation}
It follows that $\delta_0\alpha$ is smooth for $|t|<1$, say, with
\begin{equation}\label{e2b.16.8.13}
\p_t^n(\delta_0\alpha) = \int_0^1 \lambda^n \alpha^{(n+1)}(\lambda t)\,\rd \lambda.
\end{equation}
In particular if $|t|<1$, $\p_t^n\delta_0\alpha$ is bounded by $\sup|\alpha^{(n+1)}|$.  On the other hand,
\eqref{e2.16.8.13} shows that $\delta_0\alpha(t)$ decays as $|t| \to\infty$,
and by differentiating this formula, the same is true of all
derivatives.

Inserting the formula
\begin{equation}\label{e3.18.8.13}
\alpha(t) = \alpha(0) + t\delta_0 \alpha(t)
\end{equation}
into \eqref{e1a.18.8.13}, getting
\begin{eqnarray}
Z(x,\hbar;\alpha) &=& \frac{1}{\sqrt{2\pi}\hbar}\int_{-\infty}^x
e^{-t^2/2\hbar^2}\left(\alpha(0) + t\delta_0\alpha(t)\right)\,\rd t
 \nonumber \\
&=&
\frac{1}{\sqrt{2\pi}\hbar}\alpha(0)\int_{-\infty}^xe^{-t^2/2\hbar^2}\,\rd t
+
\frac{1}{\sqrt{2\pi}\hbar}
\int_{-\infty}^x 
te^{-t^2/2\hbar^2}\delta_0\alpha(t)\,\rd t  \label{e32.16.8.13} \\
&=& \alpha(0)\Phi(x/\hbar) - \frac{\hbar}{\sqrt{2\pi}}e^{-x^2/2\hbar^2}\delta_0\alpha(x)
+ 
\frac{\hbar^2}{\sqrt{2\pi\hbar}}\int_{-\infty}^xe^{-kt^2/2}[\delta_0\alpha]'(t)\,\rd t.
 \label{e31.16.8.13}
\end{eqnarray}
Here we have used the notation
\begin{equation}\label{e10.4.5.12}
\Phi(x) = \frac{1}{\sqrt{2\pi}}\int_{-\infty}^x e^{-t^2/2}\,\rd t\mbox{ so }
\Phi'(x) = \frac{1}{\sqrt{2\pi}}e^{-x^2/2}
\end{equation}
respectively for the normal distribution function and the normal density
function.  The final line \eqref{e31.16.8.13} was obtained by
integration by parts and the formula:
\begin{equation}\label{e7.4.5.12}
te^{-t^2/2h^2}= -\hbar^2 \frac{\rd}{\rd t} e^{-t^2/2h^2}.
\end{equation}
Now define
\begin{equation}\label{e9.4.5.12}
\cD \alpha  = \frac{\rd}{\rd t}\delta_0 \alpha.
\end{equation}
Then \eqref{e31.16.8.13} can be written
\begin{equation}\label{e5.18.8.13}
Z(x,\hbar;\alpha) = \alpha(0)\Phi(x/\hbar) + \hbar \delta_0\alpha(x)\Phi'(x/\hbar) +
\hbar^2 Z(x,\hbar;\cD \alpha)
\end{equation}
This forms the basis of the inductive step in obtaining a full
asymptotic expansion of $Z(x,\hbar;\alpha)$:
\begin{thm}
Suppose that $u(t)$ is smooth with all derivatives uniformly bounded on $\RR$, and let 
\begin{equation}\label{e11.4.5.12}
Z(x,\hbar;\alpha) = \frac{1}{\sqrt{2\pi}h}\int_{-\infty}^x e^{-kt^2/2} \alpha(t)\,\rd t.
\end{equation}
Then for each $N$, we have an expansion
\begin{equation}\label{e12.4.5.12}
Z(x,\hbar;\alpha) = 
\left\{\sum_{j=0}^N \hbar^{2j}\cD^j\alpha (0)\right\}\Phi(x/\hbar)
+
\left\{\sum_{j=0}^N \hbar^{2j+1}\delta_0\cD^j\alpha (x)\right\}\Phi'(x/\hbar) +
h^{2N+2}\cR_{N+1}(x,\hbar;\alpha),
\end{equation}
where the error term is given by
\begin{equation}
\cR_{N+1}(x,\hbar;\alpha) = Z(x,\hbar,\cD^{N+1}\alpha)
\end{equation}
which satisfies
\begin{equation}\label{e6.18.8.13}
|\cR_{N+1}(x,\hbar;\alpha)| \leq B_N\|\alpha\|_{C^{2N+2}}.
\end{equation}
\label{t1.18.8.13}\end{thm}
\begin{proof}
It is clear that \eqref{e5.18.8.13} yields the case $N=0$ of
\eqref{e12.4.5.12}.  It also gives the inductive step: if 
\eqref{e12.4.5.12} holds for $N$, then replacing $u$ by $\cD^N\alpha$
implies \eqref{e12.4.5.12} for $N+1$.  So it remains only to prove the estimate \eqref{e6.18.8.13}.  For this,
note that $\delta_0$ behaves as a first-order differential operator:
we have estimates of the form
\begin{equation}\label{e7.18.8.13}
\|\delta_0\alpha\|_{C^r} \leq B_r\|\alpha\|_{C^{r+1}}
\end{equation}
for each $r$, $B_r$ being a constant independent of $\alpha$. It follows
that $\cD$ satisfies estimates of the form
\begin{equation}\label{e55.16.8.13}
\|\cD \alpha\|_{C^r} \leq B'_r \|\alpha\|_{C^{r+2}}
\end{equation}
for each $r$.

To obtain the estimate of $\cR_{N+1}$ from this, 
use the 
change of variables $s = t/\hbar$ in the integral
\begin{equation}
\cR_{N+1}(x,\hbar;\alpha) = \frac{1}{\sqrt{2\pi}\hbar}\int_{-\infty}^x 
e^{-t^2/2\hbar^2}\cD^{N+1}\alpha(t)\,\rd t = 
\frac{1}{\sqrt{2\pi}}\int_{-\infty}^{x/\hbar} 
e^{-s^2/2}\cD^{N+1}\alpha(\hbar s)\,\rd s
\end{equation}
and the bound \eqref{e55.16.8.13}, iterated, to give
\begin{equation}\label{e56.16.8.13}
\sup|\cD^{N+1}\alpha| \leq B''_{N+1}\|\alpha\|_{C^{2N+2}}
\end{equation}
for some constant $B''_{N+1}$.
\end{proof}

\subsection{More general exponentials}\label{sec:laplacegeneral}

We now consider the changes needed to adapt this method to the more
general exponents appearing in 
$$F(x,\hbar;u):= \sqrt{\frac{k}{2\pi}} \int_{-\infty}^x e^{-kf(t)}\alpha(t)\,\rd t$$
as in \eqref{e1.14.8.13}.  Our approach is quite standard: we deform $f$ to its quadratic part and work out what additional terms this introduces.
So consider the
$1$-parameter family of exponents
\begin{equation}\label{e61.16.8.13}
f_\lambda(t) = \frac{ct^2}{2} + \lambda q(t)
\end{equation}
where $q(t)$ is the third-order part of $f(t)$.  Then $f_0(t)$ is a
standard quadratic, to which the above analysis can be applied, and
$f_1(t) = f(t)$ is the function we're really interested in.  Set
\begin{equation}\label{e62.16.8.13}
H(\lambda) = H(\lambda; x,\hbar,\alpha)
= \frac{1}{\sqrt{2\pi}\hbar}\int_{-\infty}^x
e^{-kf_\lambda(t)}\alpha(t)\,\rd t.
\end{equation}
With all other parameters fixed, Taylor's theorem for $H$ gives
\begin{equation}\label{e63.16.8.13}
\left|H(1) - \sum_{j=0}^p \frac{1}{j!}H^{(j)}(0)\right| \leq
C\sup_{0\leq \lambda \leq 1} H^{(p+1)}(\lambda).
\end{equation}
For simplicity take $c=1$.  Then for any $j$,
\begin{equation}\label{e64.16.8.13}
H^{(j)}(0) = \frac{(-1)^j}{\sqrt{2\pi}\hbar^{2j+1}}\int_{-\infty}^x
e^{-t^2/2\hbar^2}q(t)^j\alpha(t)\,\rd t
\end{equation}
and this integral, for fixed $j$, can be
analyzed by the techniques of the previous section.  We stress that the the third-order vanishing of $q$ implies that the integrand in \eqref{e64.16.8.13} vanishes to order $3j$ and so
contributes terms of order $\hbar^{3j}$, leaving $H^{(j)}(0)$ of order
$\hbar^j$.  In the statements that follow it is clearest to
distinguish cases according to the parity of $j$.

\begin{thm}\label{t1.29.8.13}
Suppose that $j = 2\nu + 1$ is odd.  Then there exist a sequence of
smooth functions $\alpha_{i,j}$, $i=0,\ldots, 3\nu+1$, such that
\begin{equation}\label{e31.29.8.13}
H^{(2\nu+1)}(0) = -\frac{\hbar^{2\nu+1}}{\sqrt{2\pi}}
\sum_{i=0}^{3\nu+1} \alpha_{i,j}(\hbar y)y^{2i}e^{-y^2/2} - \hbar^{2\nu +
  2}Z(x,\hbar;\cD^{3\nu+2}(q^j\alpha)).
\end{equation}
Similarly, if $j=2\nu\ge 2$ is even, then there exists a sequence of smooth
functions $\alpha_{i,j}$, $i=0,\ldots, 3\nu-1$, such that
\begin{eqnarray}\label{e32.29.8.13}
H^{(2\nu)}(0) &=& \frac{1}{\sqrt{2\pi}}\left(\hbar^{2\nu}
\sum_{i=0}^{3\nu-1} \alpha_{i,2\nu}(\hbar y)y^{2i+1}e^{-y^2/2} + \hbar^{2\nu+1}
\delta_0\cD^{3\nu}(q^{2\nu}\alpha)(\hbar y)e^{-y^2/2}
\right)  \nonumber \\
& & + 
\hbar^{2\nu +
  2}Z(x,\hbar;\cD^{3\nu+1}(q^{2\nu}\alpha)).
\end{eqnarray}
More precisely, the $\alpha_{i,j}$ are defined as follows:
\begin{equation}\label{e33.29.8.13}
\alpha_{i,2\nu+1}(x)x^{2i} = \delta_0\cD^{3\nu+1-i}(q^{2\nu+1}\alpha)
\end{equation}
and
\begin{equation}\label{e34.29.8.13}
\alpha_{i,2\nu}(x)x^{2i+1} = \delta_0\cD^{3\nu-i}(q^{2\nu}\alpha).
\end{equation}

\end{thm}

\begin{rem}
Implicit in the definitions of $\alpha_{i,j}$ in \eqref{e33.29.8.13} and
\eqref{e34.29.8.13} is that the right-hand sides of those equations
are indeed divisible by  $x^{2i}$ or $x^{2i+1}$ respectively so that
$u_{i,j}$ is smooth.
\end{rem}

\begin{rem}\label{rem:parameterlaplace}

As a final generalization, observe that the above expansion still holds if $f$ and $u$ are allowed to depend smoothly on $\hbar$ and also on auxiliary parameters.  That is, suppose that $W$ is a relatively compact subset of $\mathbb C^m$ and $f = f(t,\hbar,w)$ is smooth satisfies the above assumptions uniformly for $\hbar\ge 0$ and $w\in W$ and $\alpha=\alpha(t,\hbar,w)$ is smooth.    In fact by expanding $\alpha= \alpha_0 + \alpha_1\hbar + \cdots$ in powers of $\hbar$ we may as well consider the case of $\alpha=\alpha_0$.  Furthermore by convexity of $f$, we see that the integral is changed only by a factor of $O(k^{-\infty})$  if  $\alpha_0$ is replaced by $\chi \alpha_0$ where $\chi$ is a cut-off function that is identically $1$ on an interval that is larger than the range of $x$ of interest.  Thus we may as well assume $\alpha_0$ has compact support.\medskip

Say 
$$f(t,\hbar,w) = c t^2/2 + q(t,\hbar,w)$$
 where $c=c(\hbar,w)>0$ and $q$ has order $3$ in $t$.   Then Theorem~\ref{t1.29.8.13} gives an expansion of
  \begin{equation}
F(x,\hbar,w):=\sqrt{\frac{k}{2\pi}}\int_{-\infty}^x e^{-\hbar^{-2}f(t,\hbar,w)} \alpha(t,\hbar,w) dt\label{eq:laplacemostgeneral}
\end{equation}
in powers of $\hbar$ whose terms can be deduced from (\ref{e12.4.5.12},\ref{e62.16.8.13},\ref{e31.29.8.13}-\ref{e34.29.8.13}).  Moreover it is clear that the error term is uniform over $w\in W$.  Looking back at these expressions (e.g. \eqref{e5.18.8.13}) one sees its leading order term is given by
\begin{equation}
 F(x,\hbar,w) = \frac{\alpha(0,0,w)}{\sqrt{c(0,w)}}\Phi\left(\sqrt{c(0,w)}y\right) + O(\hbar).\label{eq:leadingtermgeneral}
\end{equation}
\end{rem}

\subsection{Lifting to the Real Blow-up}

The above formulae can be given a geometric flavour through the
introduction of the real blow-up $[\RR_x\times [0,\infty)_{\hbar},(0,0)]$.    Letting $y=x/\hbar$ as before, we have coordinates $(y,\hbar)$ on this blow-up that cover the interior of the exceptional divisor $E$.

\begin{thm}\label{t1.29.8.14}
The function $F(x,\hbar,w)$ from \eqref{eq:laplacemostgeneral} extends smoothly from the interior of $B$ to all boundary faces.
\end{thm}
\begin{proof}
This is almost obvious from the formulae we've obtained. Indeed the
form of Theorem~\ref{t1.18.8.13} and equations \eqref{e31.29.8.13} and
\eqref{e32.29.8.13} show at once (setting $x=\hbar y$) that $F$
extends smoothly to a neighbourhood of the interior of the exceptional
divisor $E$.

The blow-up $X$ has two corners $C_{\pm}$.  Here $C_{+}$ is the
intersection of $E$ with the lift of the positive real axis.  We can take $\eta := \hbar/x = 1/y$ and $x\geq 0$ as coordinates near $C_{+}$ and in these coordinates the expansion in Theorem~\ref{t1.18.8.13} has the form
\begin{align}\label{e21.29.8.13}
Z(x,\hbar;\alpha) &= 
\left\{\sum_{j=0}^N (\eta x)^{2j}\cD^j\alpha (0)\right\}\Phi(1/\eta)
+
\left\{\sum_{j=0}^N (\eta x)^{2j+1}\delta_0\cD^j\alpha (x)\right\}\Phi'(1/\eta) \\&+
(\eta x)^{2N+2}\cR_{N+1}(x,\hbar;\alpha).\nonumber
\end{align}
and because $\Phi(1/\eta)$ and $\Phi'(1/\eta)$  are smooth down to
$\eta=0$, this expression is clearly smooth for $(\eta,x)$ small and
non-negative.

Similarly, in these coordinates,
\begin{equation}\label{e22.29.8.13}
H^{(2l+1)}(0)
=
-\frac{(\eta x)^{2l}}{\sqrt{2\pi}}\sum_{i=0}^{3l+1}
\alpha_{i,j}(x))\eta^{-2i}e^{-1/2\eta^2} + \hbar^{2l+1}Z(x,\hbar;
\cD^{3l+2}\alpha)
\end{equation}
and each term in the sum is of the form
$x^{2l}\eta^{2l-2i}e^{-1/2\eta^2}$ ($i=0,\ldots, 3l+1$) which again is
smooth for $(\eta,x)$ small and non-negative.
\end{proof}

\subsection{Proofs}

To prove Theorem \ref{t1.29.8.13}  we start with a lemma:
\begin{lem}
Suppose that $g(t)$ is smooth, with all derivatives bounded, and vanishes to order $m$ at $0$: in other
words,
\begin{equation}\label{e1.18.8.13}
g^{(j)}(0) = 0\mbox{ for all  }j<m.
\end{equation}
Then $\delta_0 g$ vanishes to order $m-1$ and $\cD g$ vanishes to
order $m-2$ at $0$.
\label{l2.18.8.13}\end{lem}
\begin{proof}
The question is local to $0$, so we can use the representation
\begin{equation}\label{e2.18.8.13}
\delta_0 g (t) = \int_0^1 \p_tg(\lambda t,x)\,\rd \lambda
\end{equation}
Then by \eqref{e2b.16.8.13}, we see at once that $(\delta_0g)^{(j)}(0)
= 0$ for all $j=0,\ldots,m-1$.  This implies the corresponding
vanishing for $\cD g$.
\end{proof}

Now we can move on to the proof of the Theorem.  The proof is little
more than using Theorem~\ref{t1.18.8.13} to expand $Z(x,\hbar; q^j\alpha)$,
in combination with the information from the previous lemma.

More precisely, suppose that $j=2\nu+1$.  Then the order of vanishing
of $q^j\alpha$ is $m= 6\nu+3=2(3\nu+1)+1$. It follows from the lemma that
$\cD^{3\nu+1-i}(q^j\alpha)$ vanishes to order $2i$ for $i=0,\ldots, 3\nu+1$.  In particular the
$\alpha_{i,j}$ of \eqref{e33.29.8.13} are well-defined smooth functions.
Now apply Theorem~\ref{t1.18.8.13} with $N= 3\nu+1$. Because
$\cD^{3\nu+1-i}(q^j\alpha)(0)=0$, the coefficient of $\Phi$ in
\eqref{e12.4.5.12} is zero.  Thus we are left with
\begin{equation}\label{e35.29.8.13}
Z(x,\hbar;q^j\alpha) = \frac{e^{-y^2/2}}{\sqrt{2\pi}}\sum_{i=0}^{3\nu+1}\hbar^{2i+1}
\delta_0\cD^{i}u(q^jx)
+ \hbar^{6\nu+4}Z(x,\hbar; \cD^{3\nu+2}(q^{2\nu+1}\alpha)).
\end{equation}
Now substitute \eqref{e33.29.8.13} into \eqref{e35.29.8.13} and write
$x = \hbar y$ to get
\begin{equation}\label{e36.29.8.13}
Z(x,\hbar;q^j\alpha) = \frac{e^{-y^2/2}}{\sqrt{2\pi}}\sum_{i=0}^{3\nu+1}\hbar^{2i+1}
\alpha_{3\nu+1 -i,2\nu+1}(\hbar y)(\hbar y)^{2(3\nu+1-i)}
+ \hbar^{6\nu+4}Z(x,\hbar; \cD^{3\nu+2}(q^{2\nu+1}\alpha)).
\end{equation}
Since $H^{(2\nu+1)}(0) = - Z(x,\hbar;q^{2\nu+1}\alpha)/\hbar^{4\nu+2}$,
\eqref{e31.29.8.13} follows immediately from \eqref{e36.29.8.13}.

If $j=2\nu$ is even, then the proof follows precisely the same
lines. We apply Theorem~\ref{t1.18.8.13} with $N= 3\nu$ to expand
$Z(x,\hbar;q^{2\nu}\alpha)$ and use the fact that $q^{2\nu}\alpha$ has a zero of
order $6\nu$ at $0$.  This means that the $\alpha_{i,2\nu}$ of
\eqref{e34.29.8.13} are well-defined for $i=0,\ldots, 3\nu-1$.  We
proceed as before, using $H^{(2\nu)}(0) = Z(x,\hbar;
q^{2\nu}\alpha)/\hbar^{4\nu}$ to obtain \eqref{e32.29.8.13}.

\subsection{Remainder term}

It remains only to estimate the error term $H^{(p)}(\lambda)$ in
\eqref{e63.16.8.13}. Of course
\begin{equation}
H^{(p)}(\lambda)
= 
\frac{(-1)^p}{\sqrt{2\pi}\hbar^{2p+1}}\int_{-\infty}^x
e^{-kf_\lambda(t)}q(t)^{p+1}\alpha(t)\,\rd t.
\end{equation}
On the assumption that
\begin{equation}
f_\lambda(t) \geq t^2/4
\end{equation}
for $t\in \supp(\alpha)$ and $0\leq \lambda \leq 1$,
\begin{equation}
|H^{(p)}(\lambda)|
\leq
\frac{1}{\sqrt{2\pi}\hbar^{2p+2}}\int_{-\infty}^x
e^{-t^2/4}|q(t)|^{p+1}|\alpha(t)|\,\rd t
\leq 
\frac{1}{\sqrt{2\pi}\hbar^{2p}}\int_{-\infty}^x
e^{-kt^2/4}Q^{p+1}|t^{3p+3}\alpha(t)|\,\rd t
\end{equation}
Making the change of variables $t= \hbar s$, $x = \hbar y$, 
\begin{equation}
|H^{(p)}(\lambda)|
\leq
\frac{\hbar^{p+1}}{\sqrt{2\pi}}\int_{-\infty}^y
e^{-s^2/4}Q^{p+1}|s^{3p+3}\alpha(\hbar s)|\,\rd s
\end{equation}

\end{document}